\documentclass[11pt]{amsart}

\usepackage{amsmath}
\usepackage{amssymb}
\usepackage{amscd}

\newcommand{\Z}{\mathbb{Z}}

\newcommand{\R}{\mathbb{R}}
\newcommand{\C}{\mathbb{C}}
\renewcommand{\P}{\mathbb{P}}

\newcommand{\gltwoplus}{\widetilde{\operatorname{GL}_{2}^{+}(\R)}}
\newcommand{\surj}{\twoheadrightarrow}
\newcommand{\eps}{\epsilon}
\newcommand{\om}{\omega}
\newcommand{\si}{\sigma}
\newcommand{\de}{\delta}
\newcommand{\la}{\lambda}
\newcommand{\sub}{\subset}
\renewcommand{\AA}{{\mathcal A}}
\renewcommand{\SS}{{\mathcal S}}

\newcommand{\OO}{{\mathcal O}}
\newcommand{\DD}{{\mathcal D}}
\newcommand{\FF}{{\mathcal F}}
\newcommand{\TT}{{\mathcal T}}
\newcommand{\NN}{{\mathcal N}}
\newcommand{\CC}{{\mathcal C}}
\newcommand{\Stab}{\operatorname{Stab}}

\newcommand{\Pic}{\operatorname{Pic}}
\newcommand{\glue}{\operatorname{gl}}

\newcommand{\rk}{\operatorname{rk}}
\newcommand{\lan}{\langle}
\newcommand{\ran}{\rangle}
\newcommand{\ov}{\overline}
\newcommand{\Hom}{\operatorname{Hom}}
\newcommand{\Ext}{\operatorname{Ext}}
\newcommand{\Coh}{\operatorname{Coh}}
\newcommand{\nn}{{\bf n}}
\newcommand{\coker}{\operatorname{coker}}
\renewcommand{\ker}{\operatorname{ker}}
\newcommand{\im}{\operatorname{im}}
\newcommand{\ot}{\otimes}
\newcommand{\wt}{\widetilde}

\newcommand{\hh}{{\mathfrak h'}}

\numberwithin{equation}{section}

\newtheorem{thm}{Theorem}[section]
\newtheorem{lem}[thm]{Lemma}
\newtheorem{prop}[thm]{Proposition}
\newtheorem{cor}[thm]{Corollary}

\newenvironment{definition}[1][Definition.]{\begin{trivlist}
\item[\hskip \labelsep {\bfseries #1}]}{\end{trivlist}}

\newenvironment{remark}[1][Remark.]{\begin{trivlist}
\item[\hskip \labelsep {\bfseries #1}]}{\end{trivlist}}

\begin{document}

\title{Gluing stability conditions}
\author{John Collins and Alexander Polishchuk}
\thanks{This work of the second author was partially supported by the NSF grant DMS-0601034}

\begin{abstract} We define and study a gluing procedure for Bridgeland stability
conditions in the situation when a triangulated category has a semiorthogonal decomposition.
As an application, we construct stability conditions on the derived categories of
$\Z_2$-equivariant sheaves associated with ramified double coverings of $\P^3$. Also,
we study the stability space for the derived category of $\Z_2$-equivariant 
coherent sheaves on a smooth curve $X$, associated with a degree $2$ map $X\to Y$,
where $Y$ is another smooth curve. In the case when the genus of $Y$ is $\ge 1$ 
we give a complete description of the stability space.
\end{abstract}

\maketitle

%Remark on checking the finiteness of the generalized norm: if we can show that for all $\sigma$-stable objects E the norm of [E] in the numerical quotient of $K_0$ (that is assumed to be finitely generated) is bounded by a universal constant times |Z(E)|, then the generalized norm associated with $\sigma$ is finite.

%\vskip 0.5 in

\section*{Introduction}

Stability conditions on triangulated categories were introduced by Bridgeland in \cite{Bridgeland06}
as a mathematical formalization of Douglas' work on $\Pi$-stability in \cite{Doug1,Doug2}.
A stability condition gives a way to single out (semi)stable objects in a triangulated category $\DD$, 
generalizing Mumford's definition of stability for vector bundles. The remarkable feature of
Bridgeland's theory is that the set of (nice) stability conditions on $\DD$ has a structure of complex manifold. Hypothetically this manifold, called the {\it stability space} has some interesting geometric
structures, and in the case when $\DD$ is the derived category of coherent sheaves on a 
Calabi-Yau threefold this space should be relevant for mirror symmetry considerations 
(see \cite{Bridge-survey}). However,
at present we have a quite limited stock of examples of stability conditions, so it is important
to come up with new techniques for constructing them. Recall that a stability condition can
be described via its {\it heart}, which is an abelian category $H\sub\DD$, together with
a central charge $Z$, which is a homomorphism $K_0(\DD)\to\C$ sending every nonzero object of 
$H$ either to the (open) upper half-plane or to $\R_{<0}$. The idea to consider non-obvious abelian 
categories sitting inside derived categories is historically related to the theory of perverse sheaves,
where such abelian categories are defined using a certain gluing procedure associated
with a stratification of a topological space (see \cite{BBD}).
Thus, it seems natural to try to extend the gluing construction to stability conditions.
This is the first principal goal of the present paper. Secondly, we consider examples of the gluing
construction for stability conditions in particular geometric situations.

The notion of an abelian category sitting nicely inside a triangulated category $\DD$ is axiomatized
in \cite{BBD}. Recall that such categories appear as {\it hearts} of $t$-structures on $\DD$.
The natural setup for gluing of $t$-structures is the situation when $\DD$ has a {\it semiorthogonal
decomposition} $\DD=\lan \DD_1, \DD_2\ran$. By definition, this means
that $\DD_1$ and $\DD_2$ are triangulated subcategories in $\DD$ such that 
$\Hom(E_2,E_1)=0$ for every $E_1\in\DD_1$ and $E_2\in\DD_2$, and for every object
$E \in \DD$ there exists an exact triangle
\begin{equation}\label{semiorth-tr-eq}
E_2 \rightarrow E \rightarrow E_1\to E_2[1]
\end{equation}
with $E_1 \in \DD_{1}$, $E_2 \in \DD_{2}$. 
Assume we are given hearts of $t$-structures $H_1\sub\DD_1$ and $H_2\sub\DD_2$.
Under the additional assumption that 
\begin{equation}\label{Hom-H1-H2}
\Hom^{\leq 0}(H_1,H_2)=0
\end{equation}
the corresponding
glued heart $H$ will be the smallest full subcategory of $\DD$, closed under extensions
and containing $H_1$ and $H_2$. If we have stability conditions on $\DD_1$ and $\DD_2$
with the above hearts then we can define a central charge $Z$ on $\DD$ uniquely, so that it
restricts to the given central charges on $\DD_1$ and $\DD_2$. In order for the pair
$(H,Z)$ to determine a stability condition on $\DD$ one should check the Harder-Narasimhan
property (see section \ref{reason-sec}). 
This does not seem to follow automatically from the similar property of the original
stability conditions on $\DD_1$ and $\DD_2$. We provide two sufficient criteria for checking this 
property: the first (Proposition \ref{GluingCondition}(a)) 
imposes an additional discreteness condition on
the original stability conditions on $\DD_1$ and $\DD_2$, while the second
(Theorem \ref{reason-glue-thm}) imposes a stronger orthogonality condition
than \eqref{Hom-H1-H2}. We also check that under appropriate assumptions the gluing operation
is continuous (see Theorem \ref{a-thm} and Corollaries \ref{glue-cor}, \ref{exc-cor}).

For technical reasons we introduce the notion of a {\it reasonable} stability condition which is
slightly stronger than that of a {\it locally finite} stability condition considered by Bridgeland. 
Namely, we say that a stability condition is {\it reasonable} if the infimum of $|Z(E)|$ over
all nonzero semistable objects $E$, is positive. 
In most of our considerations we work only with reasonable stabilities.
We show in section \ref{reason-sec} that all (locally finite) stability conditions considered in the works
\cite{AB}, \cite{Bridgeland06}, \cite{Bridgeland-K3} and \cite{Macri07}
are reasonable, so this does not seem to be much of a restriction.

In the case of the semiorthogonal decomposition
associated with a full exceptional collection $(E_i)$
our gluing procedure for stabilities reduces to the construction of 
Macr\`i in \cite{Macri07} (the collection $(E_i)$ should be $\Ext$-{\it exceptional}, i.e.,
such that $\Hom^{\leq 0}(E_i,E_j)=0$ for $i\neq j$).
To get new examples of stability conditions we consider the following situation.
Let $X\to Y$ be a ramified double covering of smooth projective varieties. Then $X$ is equipped
with an involution and we can consider the derived category
$\DD=\DD_{\Z_2}(X)$ of $\Z_2$-equivariant coherent sheaves on $X$. It turns out
that this category has a semiorthogonal decomposition with one block being the category
of sheaves on $Y$ and another---sheaves on $R$, the ramification divisor in $Y$ 
(in the case of curves these semiorthogonal decompositions were considered in \cite{P-orbifold}). This allows to glue together some stability conditions for sheaves on $Y$ and $R$ into a stability condition on $\DD$. Using examples of stability conditions on surfaces constructed in \cite{AB} this gives examples of stability conditions on $\DD_{\Z_2}(X)$, where $X$ is a
ramified double cover of $\P^3$.

Finally, we study in detail the case when $X$ and $Y$ are curves. It turns out that in this case 
a lot of stability conditions on $\DD_{\Z_2}(X)$ are obtained by gluing.
In Theorem \ref{cover-thm} we describe an open simply connected subset $U$ of the stability space consisting of the stability conditions that are ``not too far" from the standard one (similar to the Mumford's stability for nonequivariant sheaves). We show that $U$ is
the universal covering of the corresponding open subset of central charges, where the group
of deck transformations is $\Z$. 
In the case when genus of $Y$ is $\geq 1$ we describe the stability space of $\DD_{\Z_2}(X)$ completely and show that it is contractible (see section \ref{class-sec}). Namely, we construct an isomorphism of the stability space with an explicit open subset of $\Sigma^n\times\C^2$, where $n$
is the number of ramification points of $X\to Y$, and
$\Sigma$ is a certain simply connected Riemann surface of parabolic type (so $\Sigma$ is isomorphic
to $\C$). This surface $\Sigma$ naturally appears as follows: we prove
that if $p\in X$ is a ramification point then a stability condition on $\DD_{\Z_2}(X)$ restricts
to a stability condition on the subcategory $\DD_p$ of objects supported at $p$
(provided $g(Y)\ge 1$). The stability
space corresponding to $\DD_p$ has form $\Sigma\times\C$, where
the central charge of $\OO_{2p}$ is given by exponentiating the projection to the second factor $\C$.
In the case when $Y=\P^1$ the stability space seems
to be more complicated due to the presence of additional exceptional objects in $\DD_{\Z_2}(X)$. 
We show in this case that our open subset $U$ contains a dense open subset consisting of stabilities
constructed from exceptional collections (see Proposition \ref{exc-P1-prop}).

\noindent
{\it Notation.} For subcategories $\AA_1,\ldots,\AA_n$ in a triangulated category $\DD$
we denote by $[\AA_1,\ldots, \AA_n]$ (resp., $\lan\AA_1,\ldots,\AA_n\ran$)
the extension-closed full subcategory (resp., triangulated subcategory) in $\DD$
generated by the $\AA_i$'s. We work with algebraic varieties over a fixed algebraically closed field
$k$. For a smooth projective variety $X$ we denote by $\DD(X)$ the bounded derived category of
coherent sheaves on $X$. For a complex number $z$ we denote by $\Re z$ and $\Im z$ its
real and imaginary part, and we call $\phi(z):=(\arg z)/\pi$ the phase of $z$.

\section{Reasonable stability conditions}
\label{reason-sec}

Throughout this section $\DD$ denotes a triangulated category.
Let us briefly recall basic definitions and results concerning local finite stability conditions on $\DD$, referring to Bridgeland's original paper \cite{Bridgeland06}
for details.

By definition, a stability condition $\sigma$ is given by a pair $(Z,P)$, where $Z:K_0(\DD)\to\C$
is a homomorphism from the Grothendieck group $K_0(\DD)$ of $\DD$, and $P$ is a slicing.
Such a slicing is given by a collection of subcategories $P(\phi)$ of semistable objects of phase $\phi$
for each $\phi\in\R$, where $\Hom(P(\phi_1),P(\phi_2))=0$ for $\phi_1>\phi_2$, and 
$P(\phi)[1]=P(\phi+1)$.
For an object $E\in P(\phi)$ we will use the notation $\phi(E)=\phi$.
Similarly to the case of vector bundles,
for each object $E$ of $\DD$ there should exist a {\it Harder-Narasimhan filtration} (HN-filtration),
i.e., a collection of exact triangles building $E$ from the semistable factors $E_1,\ldots,E_n$
(called the {\it HN-factors} of $E$),
where $\phi(E_1)>\ldots>\phi(E_n)$ ($E_1\to E$ is an analog of the subbundle of maximal phase, etc.).
For each interval $I\sub\R$ we denote by $PI\sub\DD$ the extension-closed subcategory generated by 
all the subcategories $P(\phi)$ for $\phi\in I$. For example, $P(0,1]$ denotes the subcategory corresponding to the interval $(0,1]$.

If $\sigma=(Z,P)$ is a stabiity condition then $P(0,1]$ is a heart of a bounded nondegenerate $t$-structure on $\DD$ with $\DD^{\le 0}=P(0,+\infty)$ and $\DD^{\ge 0}=P(-\infty,1]$.
We will often
refer to the abelian subcategory $P(0,1]\sub\DD$ 
as the {\it heart} of $\sigma$. By Proposition 5.3 of \cite{Bridgeland06}, to give a stability condition is
the same as to give an abelian subcategory $H\sub \DD$ 
(which should be the heart of a bounded nondegenerate $t$-structure), 
together with a homomorphism $Z:K_0(H)\to\C$ such that for every nonzero object $E\in H$ one
has either $\Im Z(E)>0$ or $Z(E)\in\R_{<0}$. These data should satisfy the Harder-Narasimhan
property, i.e., once we define (semi)stability for objects in $H$ using the slopes associated with
the function $Z$, then every object of $H$ should be equipped with an analog of the
Harder-Narasimhan filtration. Checking the Harder-Narasimhan property is often an important ingredient in constructing stability conditions 
(see section \ref{HN-sec} for examples).

A stability condition $\sigma=(Z,P)$ is called {\it locally finite} if there exists $\eta>0$ such that
for every $\phi\in\R$ the quasi-abelian category $P(\phi-\eta,\phi+\eta)$ is of finite length.
The space of all locally finite stability conditions on $\DD$ is denoted $\Stab(\DD)$.
It can be equipped with a natural topology defined as follows (see section 6 of \cite{Bridgeland06}).
For $\sigma = (Z,P) \in \Stab(\DD)$ we define a function 
$||\cdot||_{\sigma}:\Hom(K_0(\DD),\C)\to[0,+\infty]$ by
$$||U||_{\sigma}=\sup_{E \text{ semistable},E\neq 0}\frac{|U(E)|}{|Z(E)|}.$$
The basis of open neighborhoods of a locally finite stability condition
$\sigma=(Z,P)$ in $\Stab(\DD)$ consists of open subsets
$$B_{\eps}(\sigma)=\{\tau=(U,Q)\ :\  ||U-Z||_{\sigma}<\sin(\pi\eps), d(P,Q)<\eps\},$$
where $d(P,Q)$ is a natural generalized metric on the set of slicings given by
$$d(P,Q)=\inf\{\eps\in\R_{\ge 0}\ :\ Q(\phi)\sub P[\phi-\eps,\phi+\eps] \text{ for all }\phi\in\R\}.$$
Theorem 7.1 of \cite{Bridgeland06} states that for a given locally finite stability condition
$\sigma=(Z,P)$ there exists an $\eps_0>0$ such that if $0<\eps<\eps_0$ then
every central charge $Z'\in\Hom(K_0(\DD),\C)$ with $||Z'-Z||_{\sigma}<\sin(\pi\eps)$ lifts to an element
of $B_{\eps}(\sigma)$. Let us set
$$W_{\sigma} := \{U \in \Hom(K_0(\DD),\C) : ||U||_{\sigma} < \infty \}.$$ 
The linear subspaces 
$W_{\sigma}\sub\Hom(K_0(\DD),\C)$ do not change as $\sigma$ varies over a connected component
$C$ of $\Stab(\DD)$. Furthermore, the natural projection $C\to W_{\sigma}$ is a local homeomorphism
(see Theorem 1.2 of \cite{Bridgeland06})

In the case when $\DD$ is of finite type over a field one can consider the numerical Grothendieck group
$\NN(\DD)$ which is the quotient of $K_0(\DD)$ by the kernel of the Euler bilinear form on
$K_0(\DD)$ (see \cite{Bridgeland06}, 1.3). A stability condition is called {\it numerical} if the corresponding
central charge factors through $\NN(\DD)$. We denote by $\Stab_{\NN}(\DD)$ the space
of numerical locally finite stability conditions on $\DD$. The above theorem on the structure of
$\Stab(\DD)$ implies that in a neighborhood of $\sigma\in \Stab_{\NN}(\DD)$
the space $\Stab_{\NN}(\DD)$ is modeled on the linear space
$W^{\NN}_{\sigma}=W_{\sigma}\cap\Hom(\NN(\DD),\C)$.
A numerical stability condition
$\sigma$ is called {\it full} if $W^{\NN}_{\sigma}= \Hom(\NN(\DD),\C)$ 
(see \cite{Bridgeland-K3}).

The space $\Stab(\DD)$ (resp., $\Stab_{\NN}(\DD)$) is equipped with a canonical action
of the group $\gltwoplus$, which is a universal covering of the group of $2\times 2$-matrices
over $\R$ with positive determinant. 
For a real number $a$ let us denote by $R_a:\Stab(\DD)\to\Stab(\DD)$ the operation of
shifting the phase by $a$ which is part of this $\gltwoplus$-action. More explicitly, 
for $\sigma=(Z,P)$ one has
$R_a\sigma=(r_{-\pi a}\circ Z,P')$, where $P'(t)=P(t+a)$, $r_{-\pi a}$ is the rotation in $\C=\R^2$ through
the angle $-\pi a$. We refer to the transformations $R_a$ as rotations.

\begin{definition} A stability condition $\sigma=(Z,P)$ on $\DD$ is
called {\it reasonable} if 
$$\inf_{E\text{ semistable},E\neq 0} |Z(E)| >0$$
where $E$ runs over all nonzero $\sigma$-semistable objects.
\end{definition}

\begin{lem}\label{reason-lem} 
Let $\sigma=(Z,P)$ be a stability condition on $\DD$.
\begin{enumerate}
\item If $\sigma$ is reasonable then for every $0<\eta<1$ one has
$$\inf_{t\in\R, E\in P(t,t+\eta)\setminus 0} |Z(E)| >0;$$
\item $\sigma$ is reasonable if and only if for every $t$ and every
$0<\eta<1$ the point $0$ is an isolated
point of $Z(P(t,t+\eta))$;
\item If $\sigma$ is reasonable then every category $P(t,t+\eta)$ for $0<\eta<1$ is of finite length,
hence, $\sigma$ is locally finite;
\item If the image of $Z$ in $\C$ is discrete then $\sigma$ is reasonable.
\end{enumerate}
\end{lem}

\begin{proof} (1) Let 
$$c=\inf_{E\text{ semistable},E\neq 0} |Z(E)|>0.$$ 
Given an object  $E\in\P(t,t+\eta)$ let
$E_i$ be the HN-factors of $E$. Then all numbers $Z(E_i)$ (and $Z(E)$) lie in the cone $C(t,t+\eta)$
of complex numbers with phases between $t$ and $t+\eta$. Let $h:\C\to\R$ denote
the scalar product with the unit vector of phase $t+\eta/2$. Then we have
$\cos(\pi\eta/2)|z|\le h(z)\le |z|$ for all $z\in C(t,t+\eta)$. Hence, 
$$|Z(E)|\ge h(Z(E))=\sum_i h(Z(E_i))\ge\cos(\pi\eta/2)c.$$

\noindent
(2) The ``only if" part follows from (1). Conversely, assuming that $0$ is an isolated point of
$Z(P(0,3/4))$ and of $Z(P(1/2,5/4))$ we see that there is a universal lower bound for
$|Z(E)|$, where $E$ is semistable of the phase in $(0,1]$. This implies that $\sigma$ is reasonable.

\noindent
(3) This is similar to Lemma 4.4 of \cite{Bridgeland-K3}. The point is that if $h:\C\to R$ denotes the
scalar product with the unit vector of phase $t+\eta/2$ then $h(A)>c>0$ for a fixed constant $c$,
where $A$ is a nonzero object of $P(t,t+\eta)$. Since $h$ is an additive function with respect
to strict short exact sequences, the assertion follows.

\noindent
(4) This is clear.
\end{proof}

\begin{prop} Let $\Sigma$ be a connected component of $\Stab(\DD)$ containing some
reasonable stability condition. Then every $\sigma\in\Sigma$ is reasonable.
\end{prop}

\begin{proof} Let $\sigma=(Z,P)$, $\sigma'=(Z',P')$ be points of $\Sigma$.
Assume first that $\sigma'$ is reasonable, and $\sigma'\in B_{\eps}(\sigma)$, where $\eps<1/4$.
Then for every $\sigma$-semistable object $E$ of phase $t$ we have
$|Z'(E)-Z(E)|<\sin(\pi\eps)|Z(E)|$ and $E\in P'(t-\eps,t+\eps)$. Hence, by Lemma \ref{reason-lem}(1),
there exists a constant $c>0$ independent of $E$ such that $|Z'(E)|>c$. Therefore,
$$|Z(E)|>(1+\sin(\pi\eps))^{-1}|Z'(E)|>(1+\sin(\pi\eps))^{-1}c,$$
so $\sigma$ is reasonable. This shows that the set of reasonable stabilities is closed.
Conversely, assume that $\sigma$ is reasonable and $\sigma'\in B_{\eps}(\sigma)$, where $\eps$
is sufficiently small. Given a $\sigma'$-semistable object $E$ of phase $t$ we have
$E\in P(t-\eps,t+\eps)$. Let $(E_i)$ be the HN-factors of $E$ with respect to $\sigma$.
Then $E_i\in P(t-\eps,t+\eps)\sub P'(t-2\eps,t+2\eps)$. Let us denote by $h:\C\to\R$ the scalar
product with the unit vector of phase $t$. Then
$$|Z'(E)|=h(Z'(E))=\sum_i h(Z'(E_i))\ge \frac{1}{2}\sum_i |Z'(E_i)|$$
provided $\eps$ is small enough. 
But $|Z'(E_1)|>(1-\sin(\pi\eps))|Z(E_1)|$ which is bounded below by a positive constant 
depending only on $\eps$. Hence, $\sigma'$ is reasonable, so the set of reasonable stabilities is open.
\end{proof}

\begin{cor} If $\Sigma\sub\Stab(\DD)$ is a connected component containing some
stability condition such that the corresponding central charge has discrete image, then
every $\sigma\in\Sigma$ is reasonable.
\end{cor}

Note that this Corollary implies that all (locally finite) stability conditions constructed
in \cite{AB}, \cite{Bridgeland06}, \cite{Bridgeland-K3} and \cite{Macri07}
are reasonable.

\section{Gluing construction}

The general gluing construction for $t$-structures was invented in \cite{BBD}.
We start by stating a particular case of this construction (see section 3.1 of
\cite{Polishchuk06} for a related construction).
Let $\DD$ be a triangulated category equipped with a semiorthogonal decomposition
$\DD=\lan\DD_{1}, \DD_{2}\ran$.
Note that for $E \in \DD$ the
objects $E_1\in\DD_1$ and $E_2\in\DD_2$ from the exact triangle \eqref{semiorth-tr-eq}
depend functorially on $E$. Namely, 
$E_2=\rho_2(E)$,
where $\rho_2$ is the right adjoint functor to the inclusion $\DD_2\to \DD$, and $E_1=\lambda_1(E)$,
where $\lambda_1$ is the left adjoint functor to the inclusion $\DD_1\to \DD$.

\begin{lem} 
Assume we have a semiorthogonal decomposition $\DD=\lan\DD_{1}, \DD_{2}\ran$
and t-structures $(\DD_{i}^{\leq 0}, \DD_{i}^{\geq 0})$ with the hearts $H_{i}$ on $\DD_{i}$ (where $i=1,2$), such that $\Hom_{\DD}^{\leq 0}(H_1,H_2)=0$. Then
there is a t-structure on $\DD$ with the heart 
\begin{equation}
\label{H-for}
H=\{ X\in \DD \ |\ \lambda_1(X)\in H_1, \rho_2(X)\in H_2\}.
\end{equation}
With respect to this $t$-structure on $\DD$ the functors $\la_1:\DD\to\DD_1$ and
$\rho_2:\DD\to \DD_2$ are $t$-exact.
\end{lem}

\begin{proof} Set $\DD^{[a,b]}=\{ X\in\DD \|\ \lambda_1(X)\in \DD_1^{[a,b]}, \rho_2(X)\in
\DD_2^{[a,b]}\}.$ First, we have to check that $\Hom(\DD^{\le 0},\DD^{\ge 1})=0$. 
Note that our orthogonality assumption for the hearts is equivalent to
\begin{equation}\label{t-str-orth}
\Hom_{\DD}(\DD_1^{\le 0}, \DD_2^{\ge 0})=0.
\end{equation}
Now given
$X\in\DD^{\le 0}$ and $Y\in\DD^{\ge 1}$, the canonical exact triangles for $X$ and $Y$ show
that it is enough to check the vanishings 
$$\Hom(\rho_2(X),\rho_2(Y))=\Hom(\rho_2(X),\la_1(Y))=\Hom(\la_1(X),\rho_2(Y))=
\Hom(\lambda_1(X),\la_2(Y))=0.$$
The first and the fourth groups vanish since we start with $t$-structures on $\DD_1$ and $\DD_2$.
The second group vanishes by semiorthogonality, and the third---by \eqref{t-str-orth}.

Next, let us check that for every $E\in\DD$ there exists an exact triangle 
$$A\to E\to B\to A[1]$$
with $A\in\DD^{\le 0}$ and $B\in\DD^{\ge 1}$. Consider the canonical triangle
\eqref{semiorth-tr-eq}. We are going to construct
$A$ and $B$ in such a way that $A$ (resp., $B$) will be an extension of $\tau^1_{\le 0}E_1$
by $\tau^2_{\le 0}E_2$ (resp., of $\tau^1_{\ge 1}E_1$ by $\tau^2_{\ge 1}E_2$), where
$\tau^1_*$ and $\tau^2_*$ denote the truncation functors on $\DD_1$ and $\DD_2$,
respectively. First, applying the octahedron axiom to the exact triangles $E_2\to E\to E_1\to\ldots$ and
$\tau^1_{\le 0}E_1\to E_1\to \tau^1_{\ge 1}E_1\to\ldots$ we construct an exact
triangle 
$$\wt{A}\to E\to\tau^1_{\ge 1}E_1\to\ldots,$$ 
where $\wt{A}$ is an extension of
$\tau^1_{\le 0}E_1$ by $E_2$. Next, consider the exact triangle
$$\tau^2_{\le 0}E_2\to E_2\to\tau^2_{\ge 1}E_2\to\ldots.$$ 
The condition \eqref{t-str-orth} implies that
$\Hom^1(\tau^1_{\le 0}E_1, \tau^2_{\ge 1}E_2)=0$. Hence, there exists an exact triangle
$$A\to\wt{A}\to\tau^2_{\ge 1}E_2\to\ldots,$$ 
where $A$ is an extension of $\tau^1_{\le 0}E_1$ by
$\tau^2_{\le 0}E_2$. Applying the octahedron axiom once more we deduce the required statement.
 \end{proof}

Note that in the situation of the above Lemma we have $H_1\sub H$ and $H_2\sub H$.
Furthermore, every object $E \in H$ fits into an exact sequence in $H$
\begin{equation}\label{H1-H2-seq}
0 \rightarrow \rho_{2}(E) \rightarrow E \rightarrow \lambda_{1}(E) \rightarrow 0,
\end{equation}
where $\rho_2(E)\in H_2$ and $\lambda_1(E)\in H_1$.
Therefore, we also have
\begin{equation}\label{H-for2}
H=\lan H_2, H_1\ran,
\end{equation}
and $(H_{2}, H_{1})$ is a torsion pair in $H$ (see \cite{HRS} for the definition and basic properties
of torsion pairs).

Assume now that the hearts $H_1$ and $H_2$ are equipped with
stability functions $Z_i: K_0(H_i)\to \C$. Then the formula
\begin{equation}\label{Z-for}
Z(X)=Z_1(\lambda_1(X))+Z_2(\rho_2(X))
\end{equation}
defines a stability function on the glued heart $H$.

\begin{definition}
Suppose we have stability conditions $\sigma_1=(Z_1,P_1)$ on $\DD_1$ and 
$\sigma_2=(Z_2,P_2)$ on $\DD_2$, 
such that the corresponding hearts $H_1=P_1(0,1]$ and $H_2=P_2(0,1]$ satisfy $\Hom_{\DD}^{\leq 0}(H_1,H_2)=0$. Then we say that a stability condition $\sigma=(Z,P)$ on $\DD$ is {\it glued from } 
$\sigma_1$
and $\sigma_2$ if $Z$ is given by \eqref{Z-for}, and the heart $H=P(0,1]$ is given by \eqref{H-for}
(or equivalently, by \eqref{H-for2}).
\end{definition}

Note that this glued stability condition is uniquely determined by $\sigma_1$ and $\sigma_2$.
It exists if and only if the Harder-Narasimhan property for the stability function $Z$ on the glued
heart $H$ is satisfied. We have the following easy properties of glued stability conditions.

\begin{prop}\label{char-prop}
\begin{enumerate}
\item
A stability condition $\sigma=(Z,P)$ on $\DD$ is glued from $\sigma_1=(Z_1,P_1)$ on $\DD_1$
and $\sigma_2=(Z_2,P_2)$ on $\DD_2$ if and only if $Z_i=Z|_{\DD_i}$ for $i=1,2$,
$\Hom^{\leq 0}(H_1,H_2)=0$ and $H_i\sub H$ for $i=1,2$, where $H=P(0,1]$,
$H_i=P_i(0,1]$.
\item
Let $\sigma$ be a stability condition on $\DD$ with the central charge $Z$ and the
heart $H$. Assume that $H$ is glued from the hearts
$H_1\sub \DD_1$ and $H_2\sub \DD_2$, where $\Hom^{\leq 0}(H_1,H_2)=0$, so that
\eqref{H-for} holds. Then for $i=1,2$ there exists a stability condition $\sigma_i$ on $\DD_i$ 
with the heart $H_i$ and the central charge $Z_i=Z|_{\DD_i}$, so that $\sigma$ is
glued from $\sigma_1$ and $\sigma_2$.
\item
If $\sigma=(Z,P)$ is glued from $\sigma_1=(Z_1,P_1)$ and $\sigma_2=(Z_2,P_2)$ then
for every $\phi\in\R$ one has $P_1(\phi)\sub P(\phi)$ and $P_2(\phi)\sub P(\phi)$.
\end{enumerate}
\end{prop}

\begin{proof} (1) Let us observe that for every $E\in \DD$ one has the equality
$[E]=[\rho_2(E)]+[\lambda_1(E)]$ in $K_0(\DD)$, so the definition \eqref{Z-for} is
equivalent to the condition $Z|_{\DD_i}=Z_i$ for $i=1,2$. It remains to note also that
the embeddings $H_1,H_2\sub H$ imply that $\lan H_1,H_2\ran\sub H$. Since both
are hearts of nondegenerate $t$-structures this is equivalent to the equality \eqref{H-for2}.

\noindent
(2) The subcategory $H_1\sub H$ (resp., $H_2\sub H$) is exactly
the kernel of the exact functor $\rho_2:H\to H_2$ (resp., $\lambda_1:H\to H_1$).
It follows that these subcategories are closed under passing to subobjects
and quotient-objects in $H$.
This easily implies that the Harder-Narasimhan property holds for $Z|_{H_i}$ on $H_i$, $i=1,2$,
so we obtain the stability conditions on $\DD_1$ and $\DD_2$. The fact that $\sigma$ is glued
from these stabilities follows from definition.

\noindent
(3) It is enough to check this in the case when $\phi\in(0,1]$. Then this follows immediately
from the fact that $H_1$ and $H_2$ are stable under subobjects and quotient-objects in $H$.
\end{proof}

In the case of semiorthogonal decompositions associated with a full exceptional collection
$(E_1,\ldots,E_n)$ the above gluing procedure was considered
by Macr\`i in \cite{Macri07}. Namely, we can consider the semiorthogonal decomposition
$\DD=\lan\lan E_1\ran,\ldots,\lan E_n\ran\ran$, and equip $\lan E_i\ran$ with the $t$-structure 
for which $E_i$ belongs to the heart. Then our
orthogonality condition on the hearts reduces to the condition
that the collection is $\Ext$-exceptional, i.e., $\Hom^{\leq 0}(E_i,E_j)=0$ for $i<j$, and the glued heart
is $H=[E_{1}, \cdots, E_{n}]$. 
We say that a stability condition $\sigma=(Z,P)$ on $\DD$ is \emph{glued from an $\Ext$-exceptional collection} $(E_{1}, \cdots, E_{n})$ if $P(0,1]=H$.
Note that in this case the Harder-Narasimhan property is automatically satisfied
for any stability function on $H$. We will generalize this in Proposition \ref{GluingCondition}.

%\vskip 0.5 in

\section{Harder-Narasimhan property and gluing of stability conditions}\label{HN-sec}

In this section we show how to check the Harder-Narasimhan property for the
glued stability function under different sets of additional assumptions.

We start with the following basic criterion which is a slight generalization of Proposition 2.4 of
\cite{Bridgeland06} (the proof is the same as in {\it loc. cit.}, using properties of quasi-abelian
categories). Recall that $\phi(z)$ denotes the phase of $z\in\C$.

\begin{prop}\label{HN-prop} 
Suppose $\mathcal{A}$ is a quasi-abelian category with a stability function 
$Z: K_0(\mathcal{A}) \rightarrow \C$. Assume that for a pair of $Z$-semistable objects
$E,F\in\mathcal{A}$ such that $\phi(E)>\phi(F)$ one always has $\Hom_{\AA}(E,F)=0$,
where we denote $\phi(E):=\phi(Z(E))$.
Assume also that the following chain conditions are satisfied:
\begin{enumerate}
\item
there are no infinite sequences of strict monomorphisms in $\mathcal{A}$
$$\cdots \subset E_{j+1} \subset E_{j} \subset \cdots \subset E_{2} \subset E_{1}$$
with $\phi(E_{j+1}) > \phi(E_{j})$ for all $j$,
\item
there are no infinite sequences of strict epimorphisms in $\mathcal{A}$
$$E_{1} \surj E_{2} \surj \cdots \surj E_{j} \surj E_{j+1} \surj \cdots$$
with $\phi(E_{j}) > \phi(E_{j+1})$ for all $j$.
\end{enumerate}
Then $Z$ has the Harder-Narasimhan property on $\mathcal{A}$.
\end{prop}

Quasi-abelian categories often arise as follows. Consider an abelian category $\AA$
equipped with a torsion pair $(\TT,\FF)$. The both $\TT$ and $\FF$ are quasi-abelian
categories. Indeed, this follows from Lemma 1.2.34 of \cite{Schn}, using the tilted
abelian category $\AA^t$. For example, to check that $\TT$ is quasi-abelian
we use the fact that the embedding of $\TT$ into
$\AA$ is stable under quotients, while the embedding of $\TT$ into $\AA^t$
is stable under subobjects.

\begin{lem}\label{Z-torsion-lem}
Let $\AA$ be an abelian category equipped with a torsion pair $(\TT,\FF)$. Suppose
$Z$ is a stability function on $\AA$ such that for any nonzero $T\in\TT$ and $F\in\FF$ 
one has $\phi(T)>\phi(F)$ (where as before we set $\phi(F):=\phi(Z(F))$). 
Let $Z|_{\TT}$ and $Z|_{\FF}$ be the stability functions
on the exact categories $\TT$ and $\FF$ induced by $Z$.
Then every $Z|_{\TT}$-semistable object of $\TT$
(resp., $Z|_{\FF}$-semistable object of $\FF$) is $Z$-semistable as an object of $\AA$.
\end{lem}

\begin{proof}
We consider only the case of a $Z|_{\TT}$-semistable object $T\in\TT$ (the second
case is similar). Suppose $T$ is not $Z$-semistable as an object of $\AA$. Then there exists
a subobject $A\sub T$ such that $\phi(A)>\phi(T)$. Consider the canonical exact sequence
\begin{equation}\label{torsion-ex-seq}
0\to T(A)\to A\to F(A)\to 0
\end{equation}
with $T(A)\in\TT$, $F(A)\in\FF$. By the assumption either $\phi(T(A))> \phi(F(A))$ or one of
the objects $T(A)$, $F(A)$ is zero. Note that $T(A)\neq 0$, since otherwise $A$ would be an object
of $F$, so the inequality $\phi(A)>\phi(T)$ would be impossible.
It follows that $\phi(T(A))\ge\phi(A)>\phi(T)$. Thus, we found a destabilizing subobject 
$T(A)\sub T$ (the quotient is automatically in $\TT$ since $\TT$ is always closed under quotients).
\end{proof}

\begin{prop}\label{Z-torsion-prop}
Keep the assumptions of Lemma \ref{Z-torsion-lem}. Assume that 
both $(\TT, Z|_{\TT})$ and $(\FF,Z|_{\FF})$ satisfy chain conditions (1) and (2) from
Proposition \ref{HN-prop}. Then $Z$ has the Harder-Narasimhan property on $\AA$.
\end{prop}

\begin{proof}
Suppose we have a pair of $Z|_{\TT}$-semistable objects $E,F\in\TT$ such that $\phi(E)>\phi(F)$.
Then by Lemma \eqref{Z-torsion-lem}, $E$ and $F$ are still semistable viewed as objects of
the {\it abelian} category $\AA$ with the stability function $Z$. Hence, $\Hom(E,F)=0$.
Therefore, by Proposition \ref{HN-prop}, the Harder-Narasimhan property holds for $(\TT, Z|_{\TT})$.
The same argument works for $(\FF, Z|_{\FF})$. Now given an object $E\in\AA$ we can sew
together the HN-filtrations of the objects $T(A)$ and $F(A)$ from the canonical exact sequence
\eqref{torsion-ex-seq}. It remains to apply Lemma \ref{Z-torsion-lem} again to see that we get
a HN-filtration of $E$ in this way.
\end{proof}

The following Lemma is a more precise version of Proposition 5.0.1 of \cite{AP}.

\begin{lem}\label{Noeth-lem} 
(a) Let $Z$ be a stability function on an abelian category $\AA$. Assume that $0$
is an isolated point of $\Im Z(\AA)\sub\R_{\geq 0}$, and that the category
$\AA_0=\{A\in\AA\ |\ \Im Z(A)=0\}$ is Noetherian. Then $Z$
satisfies the Harder-Narasimhan property on $\AA$ if and only if $\AA$ is Noetherian.

\noindent
(b) Let $\sigma=(Z,P)$ be a stability condition on $\DD$ with Noetherian heart
$P(0,1]$. Assume that $0$ is an isolated point of $\Im Z(P(0,1))\sub\R_{\ge 0}$.
Then the category $P(0,1)$ is of finite length. Also, $\sigma$ is reasonable if and only if
$0$ is an isolated point of $Z(P(1))\sub\R_{\le 0}$. 
\end{lem}

\begin{proof} (a) Assume first that $\AA$ is Noetherian. Then condition (2) of
Proposition \ref{HN-prop} is automatic. To check condition (1)
we observe that
if $E \rightarrow F$ is a destabilizing inclusion in $\AA$ then $\Im Z(E) < \Im Z(F)$. 
Indeed, we have either $\Im Z(F/E) > 0$ or $\Re Z(F/E) < 0$. But in the latter case
the phase of $Z(E)$ would be smaller than that of $Z(F)$. 
Thus, if we have a chain 
\begin{equation}\label{incl-seq}
\cdots \subset E_{j+1} \subset E_{j} \subset \cdots \subset E_{2} \subset E_{1}
\end{equation}
of destabilizing inclusions in $\AA$ then the sequence $(\Im Z(E_j))$ is strictly decreasing. But this
implies that $\Im Z(E_j/E_{j+1})$ tends to $0$ which is a contradiction.
Conversely, assume $Z$ satisfies the Harder-Narasimhan property.
To check that $\AA$ is Noetherian we have to check that every
sequences of quotients in $\AA$
\begin{equation}\label{surj-seq}
E_1\surj E_2\surj E_3\surj\cdots
\end{equation}
stabilizes. Note that in this situation the sequence $(\Im Z(E_i))$ is decreasing, so it has
to stabilize. Without loss of generality we can assume that the sequence $(\Im Z(E_i))$ is
constant. Then the kernel $K_i$ of $E_1\to E_i$ belongs to $\AA_0$. Since $Z$ satisfies the Harder-Narasimhan property, there exists a maximal subobject $F\sub E_1$ such that $F\in\AA_0$. Then the kernels
$K_i$ form an increasing chain of subobjects in $F$. Since $\AA_0$ is Noetherian, this
sequence stabilizes, so the original sequence $(E_i)$ also stabilizes.
It remains to check that in this situation $\AA_{>0}$ is Artinian. But a sequence of
inclusions \eqref{incl-seq} with $\Im Z(E_j/E_{j+1})>0$ is impossible since $\Im Z(E_j/E_{j+1})$
would tend to zero.

\noindent
(b) To see that $P(0,1)$ is of finite length we observe that any increasing
chain of admissible inclusions in $P(0,1)$ stabilizes since $\AA=P(0,1]$ is Noetherian.
Also, if we have a chain \eqref{incl-seq} of admissible proper inclusions in $P(0,1)$ then
the sequence $\Im Z(E_j)$ is strictly decreasing, which is impossible. Under our assumptions
$|Z(E)|$ is bounded below by some positive constant, where $E$ runs through nonzero
semistable objects in $P(0,1)$. Thus, $\sigma$ is reasonable if and only if 
$$\inf_{E\in P(1)\setminus 0} |Z(E)|>0.$$
\end{proof}

\begin{prop}
\label{GluingCondition}
Let $(\DD_1, \DD_2)$ be a semiorthogonal decomposition of a triangulated category $\DD$, 
and let $\sigma_1=(Z_1,H_1)$ and $\sigma_2=(Z_2,H_2)$ be a pair of locally finite stability conditions
on $\DD_1$ and $\DD_2$, respectively.  
Assume that $\Hom_{\DD}^{\le 0}(H_1, H_2) = 0$, and let $H$ be the heart in $\DD$ glued from $H_1$ and $H_2$. As before, consider the stability function $Z = Z_1 \lambda_1 + Z_2 \rho_2$ on $H$. 
Assume in addition that one of the following two conditions hold:

\noindent
(a) $0$ is an isolated point of $\Im Z_i(H_i)\sub\R_{\geq 0}$ for $i=1,2$;

\noindent
(b) $\Hom_\DD^{\le 1}(H_1,P_2(0,1))=0$.

Then $Z$ has the Harder-Narasimhan property on $H$.
Furthermore, in case (a) 
the category $P(0,1)$ for the glued stability condition $\sigma=(Z,P)$ is of finite length.
In case (b) the stability condition $\sigma$ is locally finite.
\end{prop}

\begin{proof}
First, assume that (a) holds. Then it is easy to see that 
$0$ is an isolated point of $\Im Z(H)\sub \R_{\geq 0}$.
Also, by Lemma \ref{Noeth-lem}(a), both categories $H_1$ and $H_2$ are Noetherian (the condition
on $\AA_0$ in this Lemma follows from the assumption that $\sigma_i$'s are locally finite).
Using the exact functors $\lambda_1:H\to H_1$ and $\rho_2:H\to H_2$ we easily deduce that
$H$ is Noetherian. Now the assertion follows by applying Lemma \ref{Noeth-lem}(a) again.

\noindent
(b) In this case for every $t\in (0,1]$ let us define the subcategory $P(t)\sub H$ by
$$P(t):=\{E\in H\ |\ \lambda_1(E)\in P_1(t), \rho_2(E)\in P_2(t)\}.$$ 
Note that each object of $P(t)$ is an extension of an object in $P_2(t)$ by an
object in $P_1(t)$. It is enough for every $E\in H$ to construct the HN-filtration with respect
to this slicing. We start with the canonical extension
$$0\to E_2\to E\to E_1\to 0$$
where $E_2=\rho_2(E)\in H_2$ and $E_1=\lambda_1(E)\in H_1$. Consider also the
canonical exact sequences
$$0\to A_i\to E_i\to B_i\to 0$$
with $A_i\in P_i(1)$ and $B_i\in P_i(0,1)$ for $i=1,2$.
Since $\Hom^1(E_1,B_2)=0$ by assumption, we get a splitting $E\to B_2$ which
gives rise to an exact sequence
$$0\to A_2\to E\to B_2\oplus E_1\to 0$$
Let $E(1)\sub E$ be the preimage of $A_1\sub E_1\sub B_2\oplus E_1$.
Then $E(1)$ is an extension of $A_1$ by $A_2$, so $E(1)\in P(1)$.
Also, $E/E(1)\simeq B_1\oplus B_2$, so we get the required filtration by using
the HN-filtrations on $B_1$ and $B_2$. 
The obtained glued stability has the property that $\lambda_1(P(a,b))\sub P_1(a,b)$
and $\rho_2(P(a,b))\sub P_2(a,b)$. This easily implies that it is locally finite.
\end{proof}

\begin{remark} We do not know how to check local finiteness of the glued stability condition in
Proposition \ref{GluingCondition}(a) without imposing additional assumptions.
\end{remark}

If we work with reasonable stability conditions, we can prove the existence of the glued
stability conditions under a slightly stronger orthogonality assumption.

\begin{thm}\label{reason-glue-thm}
Let $(\DD_{1}, \DD_{2})$ be a semiorthogonal decomposition of a triangulated category $\DD$.
Suppose $(\sigma_{1}, \sigma_{2})$ is a pair of reasonable 
stability conditions on $\DD_1$ and $\DD_2$, respectively, with the slicings $P_i$ and
central charges $Z_i$ ($i=1,2$), and let $a$ be a real number in $(0,1)$. Assume the following
two conditions hold:
\begin{enumerate}
\item $\Hom_\DD^{\leq 0}(P_1(0,1],P_2(0,1])=0$;
\item $\Hom_{\DD}^{\leq 0}(P_{1}(a,a+1], P_{2}(a,a+1]) = 0$;
\end{enumerate}
Then there exists a stability $\sigma$ glued from $\sigma_1$ and $\sigma_2$.
Furthermore, $\sigma$ is reasonable.
\end{thm}

\noindent
{\it Proof of Theorem \ref{reason-glue-thm}.}
Let $H\sub\DD$ be the heart glued from $P_1(0,1]$ and $P_2(0,1]$ and let
$(\DD^{\le 0},\DD^{\ge 0})$ denote the corresponding $t$-structure.
Using the second condition we can construct a $t$-structure on $\DD$ with the
heart
$$H_a=\langle  P_1(a,a+1], P_2(a,a+1]\rangle.$$
One immediately checks that $H\sub \lan H_a, H_a[-1]\ran$ and 
$H_a\sub\lan H[1], H\ran=\DD^{[-1,0]}$. Now for every $E\in H$ consider the
canonical triangle
$$A\to E\to B\to A[1]$$
with $A\in H_a$ and $B\in H_a[-1]$. We claim that $A$ and $B$ belong to $H$.
Indeed, we have $A\in H_a\sub\DD^{\le 0}$. On the other hand,
$A$ is an extension of $E$ by $B[-1]\in H_a[-2]$, so $A\in \DD^{\ge 0}$. Hence, $A\in H$.
Similarly, $B\in H_a[-1]\sub \DD^{\ge 0}$, and also $B\in\DD^{\leq 0}$ as an
extension of $A[1]\in H_a[1]$ by $E$.
Therefore, if we set
$$P(0,a]=\{E\in\DD\ |\ \lambda_1(E)\in P_1(0,a], \rho_2(E)\in P_2(0,a]\},$$
\begin{equation}\label{Pa1-eq}
P(a,1]=\{E\in\DD\ |\ \lambda_1(E)\in P_1(a,1], \rho_2(E)\in P_2(a,1]\},
\end{equation}
then $(P(a,1], P(0,a])$ is a torsion pair in $H$. 
Next, let $Z$ be the glued central charge given by \eqref{Z-for}.
Then we have $\phi(Z(E))\leq a$ for $E\in P(0,a]$, while $\phi(Z(E))>a$ for 
$E\in P(a,1]$.
Also, since $\sigma_1$ and $\sigma_2$ are reasonable, by Lemma \ref{reason-lem}(3),
the categories $P_1(0,a]$ and $P_2(0,a]$ (resp., $P_1(a,1]$ and $P_2(a,1]$)
are of finite length.  
This implies that both $P(0,a]$ and $P(a,1]$ are also of finite length. 
Therefore, we can apply Proposition \ref{Z-torsion-prop} to the torsion pair $(P(a,1], P(0,a])$
in $H$ to derive that the Harder-Narasimhan property holds for 
$(Z,H)$. Hence, we have the corresponding stability condition $\sigma$ on $\DD$. 
It follows from the definition of $P(0,a]$ and $P(a,1]$ that $0$ is an isolated point
of $Z(P(0,a])$ and of $Z(P(a,1])$. This immediately implies that $\sigma$ is reasonable.
\qed

\begin{remark}
It may not be easy in general to determine for a particular pair of stabilities $\sigma_{1}, \sigma_{2}$ 
with $\Hom_\DD^{\leq 0}(P_1(>0),P_2(\leq 1)=0$
whether there exists $a \in (0,1)$ such that \[\Hom_{\DD}^{\leq 0}(P_{1}(>a), P_{2}(\leq a+1) = 0.\] 
However, in the following two cases this is automatic.

\noindent 1.
If there exists $\phi > 0$ such that $P_{2}(0,\phi] = \{ 0\}$ then any $a \in (0,\phi]$ works,
since in this case $P_2(\leq a+1)=P_2(\leq 1)$. 
For instance, this condition is satisfied when $P_2(0,1]$ is of finite length and has finite number of simple objects. 

\noindent 2.
If there exists $\phi < 1$ such that $P_{1}(\phi,1] = \{ 0 \}$ then any $a \in (\phi,1] $ works, since
in this case $P_1(>-a)=P_1(>0)$. 
For example, this condition holds when $P_{1}(0,1]$ is of finite length with finite number of simple objects and $P_{1}(1) = \{ 0 \}$.
\end{remark}

%\vskip 0.5 in

\section{Continuity of gluing}

Let us recall the following basic result.

\begin{lem}\label{simple-distance-lem}
(Lemma 6.4 of \cite{Bridgeland06})
Suppose $\sigma = (Z,P)$ and $\tau = (Z,Q)$ are stability conditions on $\mathcal{D}$ with the same central charge $Z$. Suppose also that $d(P,Q) < 1$. Then $\sigma = \tau$.
\end{lem}

We start with the observation that the condition $d(P,Q)<1$ in the above Lemma can be weakened
and use this to give a nice criterion for determining when two stability conditions are close
(part (b) of the following Proposition).

\begin{prop}\label{FormalDistance}
Let $\sigma_1=(Z_1,P_1)$ and $\sigma_2=(Z_2,P_2)$ be stability conditions on $\DD$. 

\noindent
(a) Assume that
\begin{enumerate}
\item
$Z_{1} = Z_{2}$ and
\item
$P_{1}(0,1] \subset P_{2}(-1,2]$.
\end{enumerate}
Then $\sigma_{1} = \sigma_{2}$.

\noindent
(b) Assume that $\sigma_1$ is locally finite. There exists $\epsilon_0>0$ such that
if for some $0 < \epsilon < \epsilon_0$
one has
\begin{enumerate}
\item
$||Z_1-Z_2||_{\sigma_1} <\sin(\pi\epsilon)$ and
\item
$P_{2}(0,1] \subset P_{1}(-1+\epsilon, 2-\epsilon],$
\end{enumerate}
then $\sigma_2\in B_{\epsilon}(\sigma_1)$. 
\end{prop}

\begin{proof}
(a) First, using properties of $t$-structures we can easily deduce that $P_2(0,1]\sub P_1(-1,2]$.
Now given $E \in P_{1}(0,1]$, there is an exact triangle 
$$F \to E \to G\to F[1]$$ 
with $F \in P_{2}(1,2]$ and $G \in P_{2}(-1,1]$. Observe that $F \in P_{1}(>0)$ and $G \in P_{1}(\leq 2)$. Since $F$ is an extension of $E$ by $G[-1]$, we derive that $F \in P_{1}(0,1]$. But the intersection $P_{1}(0,1] \cap P_{2}(1,2]$ is trivial (since $Z_1=Z_2$), so $F = 0$. This proves that $E\in P_2(-1,1]$.

Next, consider an exact triangle
$$F \to E \to G\to F[1]$$
with $F \in P_{2}(0,1]$ and $G \in P_{2}(-1,0]$. Observe that $F \in P_{1}(>-1)$ and $G \in P_{1}(\leq 1]$. Since $G$ is an extension of $F[1]$ by $E$, we get $G \in P_{2}(-1,0] \cap P_{1}(0,1] = \{ 0 \}$. Therefore, $P_{1}(0,1]\sub P_{2}(0,1]$. Since these are both hearts of bounded $t$-structures, they have
to be equal, so $\sigma_{1} = \sigma_{2}$.

\noindent (b)
Let $\sigma=(Z_2,P)$ be the unique stability in $B_{\eps}(\sigma_1)$
lifting the central charge $Z_2$ 
---it exists by our assumption that $||Z_2-Z_1||_{\sigma_1}<\sin(\pi\epsilon)$ (using Theorem 7.1 of
\cite{Bridgeland06}). Then
$$P_2(0,1]\subset  P_1(-1+\eps,2-\eps]\subset P(-1,2].$$
By part (a), this implies that $\sigma=\sigma_2$.
\end{proof}

Now we can show that the gluing construction of Theorem
\ref{reason-glue-thm} is continuous.

\begin{thm}\label{a-thm} 
Let $(\DD_{1}, \DD_{2})$ be a 
semiorthogonal decomposition in a triangulated category $\DD$.
For a real number $a\in (0,1)$ let
$S(a)\sub\Stab(\DD_1)\times\Stab(\DD_2)$ denote the subset of $(\sigma_1,\sigma_2)$ 
such that $\sigma_1$ and $\sigma_2$ are reasonable stability conditions satisfying
\begin{enumerate}
\item $\Hom_\DD^{\leq 0}(P_1(0,1],P_2(0,1])=0$,
\item $\Hom_{\DD}^{\leq 0}(P_{1}(a,a+1], P_{2}(a,a+1]) = 0$.
\end{enumerate}
Let 
$\glue:S(a)\to\Stab(\DD)$ be the map associating to $(\sigma_1,\sigma_2)$ the corresponding
glued stability condition $\sigma$ on $\DD$ (see Theorem \ref{reason-glue-thm}).
Then the map $\glue$ is continuous on $S(a)$. 
\end{thm}
\begin{proof}
Let $\sigma_i=(Z_i,P_i)$, $\sigma'_i=(Z'_i,P'_i)$ be stabilities
on $\DD_i$ for $i=1,2$, such that $(\sigma_1,\sigma_2)$ and $(\sigma'_1,\sigma'_2)$ are points of $S(a)$,
and let us denote by $\sigma=(Z,P)$ and $\sigma'=(Z',P')$ the corresponding glued stability
conditions. Assume that $\sigma'_i\in B_{\de}(\sigma_i)$ for $i=1,2$. Then for 
$\eps\ge\delta$ we have
\begin{eqnarray*}
P(0,1] = \langle P_{1}(0,1], P_{2}(0,1]\rangle & \sub & \langle P'_{1}(-\epsilon,1+\epsilon], 
P'_{2}(-\epsilon,1+\epsilon]\rangle \\ 
& \sub & P'(-\epsilon,1+\epsilon].
\end{eqnarray*}
Thus, we can deduce the required continuity from 
Proposition \ref{FormalDistance}(b), once we show that $||Z-Z'||_{\sigma} \leq\sin(\pi\epsilon)$ 
provided $\de$ is small enough. 
Let $\phi \in (0,1]$ and $E \in P(\phi).$ We have to prove that
\[ |Z(E)-Z'(E)| \leq |Z(E)|\sin(\pi \epsilon). \] 
Assume first that $\phi \in (a,1]$. Let $h: \C \rightarrow \R$ denote the scalar product with
the unit vector of phase $\frac{a+1}{2}$. 
Then there exists a positive constant $c$ (depending only on $a$) such that
\[ h(z) \leq |z| \leq c\cdot h(z), \] 
for all nonzero complex numbers $z$ with phase $\theta$, where $a\leq \theta \leq 1$. 

Let $F_{1}, \cdots, F_{n}$ (resp., $G_{1}, \cdots, G_{m})$ be the HN-factors of $\lambda_{1}(E)$ (resp., $\rho_{2}(E)$) with respect to $\sigma_1$ (resp., $\sigma_2$). 
Then we have
\begin{eqnarray*}
|Z(E)-Z'(E)| & \leq & |Z_{1}(\lambda_{1}E)-Z_{1}'(\lambda_{1}E)| + |Z_{2}(\rho_{2}E)-Z_{2}'(\rho_{2}E)| \\
& \leq & \sum_{i=1}^{n} |Z_{1}(F_{i})-Z_{2}'(F_{i})| + \sum_{j=1}^{m} |Z_{2}(G_{j})-Z_{2}'(G_{j})| \\
& \leq & \sin(\pi \delta)\left[\sum_{i=1}^{n} |Z_{1}(F_{i})|+\sum_{j=1}^{m} |Z_{2}(G_{j})|\right].
\end{eqnarray*}
Recall that by \eqref{Pa1-eq}, we have $\la_1(E)\in P_1(a,1]$ and $\rho_2(E)\in P_2(a,1]$.
Hence, all the numbers $Z_1(F_i)$ and $Z_2(G_j)$ have phases between $a$ and $1$, so we derive
\begin{eqnarray*}
|Z(E)-Z'(E)| & \leq 
c \sin(\pi \delta)[\sum_{i=1}^{n} h(Z_{1}(F_{i}))+\sum_{j=1}^{m} h(Z_{2}(G_{j}))]\\
& = c \sin(\pi \delta)h(Z(E))\leq c \sin(\pi \delta)|Z(E)|.
\end{eqnarray*}
So $\delta$ must be chosen to satisfy the relation $c \sin(\pi \delta) < \sin(\pi \epsilon).$ A similar argument covers the case of objects $F \in P(0,a]$ and imposes a second condition that $c' \sin(\pi \delta) < \sin(\pi \epsilon)$ for some positive constant $c'$, depending only on $a$. Given $\delta$ satisfying both conditions, it follows that $$||Z-Z'||_{\sigma} \leq \sin(\pi \epsilon).$$
\end{proof}

The following Corollary describes an open subset of pairs of stabilities that can be glued,
obtained by imposing a stronger orthogonality assumption on $(\sigma_1,\sigma_2)$.

\begin{cor}\label{glue-cor}
Let $U\sub\Stab(\DD_1)\times\Stab(\DD_2)$ denote the set of pairs 
of  reasonable stabilities
$(\sigma_1=(Z_1,P_1)$ and $\sigma_2=(Z_2,P_2))$
such that  for some $\eps>0$ one has
$$\Hom_{\DD}^{\le 0}(P_1(-\eps,1],P_2(0,1+\eps))=0.$$
Then $U$ is open and the gluing map $\glue:U\to\Stab(\DD)$ 
is continuous.
\end{cor}

\begin{proof} Note that our assumption on $(\sigma_1,\sigma_2)$ is equivalent to
$$\Hom_\DD(P_1(-\eps,+\infty),P_2(-\infty,1+\eps))=0.$$
For each $\eps>0$
let us denote by $T_{\eps}$ the set of pairs  $(\sigma_1,\sigma_2)$ satisfying this condition.
Note that $U=\cup_{\eps>0}T_{\eps}$. Now to check that
$U$ is open suppose we have $(\sigma_1,\sigma_2)\in T_{\eps}$.
Given a pair $(\sigma'_1=(Z'_1,P'_1),\sigma'_2=(Z'_2,P'_2))$,
such that $\sigma'_i\in B_{\de}(\sigma_i)$, for $i=1,2$, where $0<\de<\eps$, we have
$P'_1(>-\eps+\de)\sub P_1(>-\eps)$ and $P'_2(<1+\eps-\de)\sub P_1(<1+\eps)$.
Hence, $(\sigma'_1,\sigma'_2)$ belongs to $T_{\eps-\de}$. It remains to apply Theorem
\ref{a-thm}.
\end{proof}

On the other hand, in the situation when $\DD_1$ is generated by an exceptional object,
we have the following result that will be used later.

\begin{cor}\label{exc-cor} 
Let $(\DD_{1}, \DD_{2})$ be a 
semiorthogonal decomposition in a triangulated category $\DD$.

\noindent (i)
Assume that $\DD_1$ is generated by an exceptional object $E_1$, and
$H_2\sub\DD_2$ is a heart of some bounded $t$-structure on $\DD_2$, such that
$\Hom^{\leq -1}_{\DD}(E_1,H_2)=0$. Let $S_2\sub\Stab(\DD_2)$ denote the set
of reasonable stability conditions $\sigma_2=(Z,P)$ with $P(0,1]=H_2$.
On the other hand, let $R_1\sub\Stab(\DD_1)$ denote the set of stability conditions
such that the phase of $E_1$ is $<0$. Then there is continuous gluing map
$R_1\times S_2\to\Stab(\DD)$.

\noindent (ii) 
Similarly, assume that $\DD_2$ is generated by an exceptional object $E_2$, and
$H_1\sub\DD_1$ is a heart of some bounded $t$-structure on $\DD_2$, such that
$\Hom^{\leq -1}_{\DD}(H_1,E_2)=0$. Let $S_1\sub\Stab(\DD_1)$ denote the set
of reasonable stability conditions with the heart $H_1$, and 
let $R_2\sub\Stab(\DD_2)$ denote the set of stability conditions
such that the phase of $E_2$ is $>1$. Then there is continuous gluing map
$S_1\times R_2\to\Stab(\DD)$.
\end{cor}

\begin{proof}
We will only consider (i) since the proof of (ii) is analogous. Let $R_1(\eps)\sub\Stab(\DD_1)$ denote
the set of stability conditions such that the phase of $E_1$ is $<-\eps$. 
It is enough to check that for every $\eps>0$ one has $R_1(\eps)\times S_2\sub S(1-\eps)$,
where $S(1-\eps)\sub\Stab(\DD_1)\times\Stab(\DD_2)$ is the subset considered in Theorem \ref{a-thm}
for $a=1-\eps$.
Note that $P_1(0,1]=\lan E_1[n]\ran$, where $n$ is determined by the condition that the phase
of $E_1$ is in the interval $(-n,-n+1]$. Hence, $n\ge 1$, so the condition 
$\Hom^{\le 0}(P_1(0,1],H_2)=0$ is satisfied. Similarly, $P_1(-\eps, 1-\eps]=\lan E_1[m]\ran$,
where $m\ge 1$. Hence, $\Hom^{\le 0}(P_1(-\eps,1-\eps], P_2(\le 1))=0$ which implies the
condition (2) of Theorem \ref{a-thm} for $a=1-\eps$.
\end{proof} 

\section{Semiorthogonal decompositions associated with double coverings}

%\vskip 0.2 in

Let $\pi: X \to Y$ be a double covering of smooth projective varieties $X$ and $Y$, ramified along a smooth divisor $R$ in $Y$. Then we have an action of $\Z_2$ on $X$ such that the nontrivial element acts
by the corresponding involution $\tau:X\to X$.
Let us denote by $\DD_{\Z_2}(X)$ the corresponding bounded derived category of $\Z_2$-equivariant
coherent sheaves on $X$. We denote by $\zeta$ the nontrivial character of $\Z_2$.
Note that $\tau$-invariant stability conditions on $\DD(X)$ correspond to stability conditions 
on $\DD_{\Z_2}(X)$
that are invariant under the autoequivalence $F\mapsto F\otimes\zeta$ (see \cite{MMS} or \cite{Polishchuk06}). 
Below we will show how to construct stability conditions on $\DD_{\Z_2}(X)$ starting from a pair of
 stability conditions on $\DD(Y)$ and on $\DD(R)$, satisfying certain assumptions.

Let us denote by $i:R\to X$ (resp., $j:R\to Y$) the closed embedding of the ramification divisor
into $X$ (resp., $Y$).
For every sheaf $F$ on $R$ we equip $i_*F$ with the trivial $\Z_2$-equivariant structure.
This gives a functor $i_*:\DD(R)\to \DD_{\Z_2}(X)$. On the other hand, for a coherent sheaf $F$ on $Y$
we have a natural $\Z_2$-equivariant structure on $\pi^*F$, so we obtain a functor
$\pi^*:\DD(Y)\to \DD_{\Z_2}(X)$.

\begin{thm}\label{semiorth-thm} 
The functors $i_*:\DD(R)\to \DD_{\Z_2}(X)$ and $\pi^*:\DD(Y)\to \DD_{\Z_2}(X)$
are fully faithful.
We have two canonical semiorthogonal decompositions of $\DD_{\Z_2}(X)$:
$$
\DD_{\Z_{2}}(X) = \langle \pi^{*}\DD(Y), i_{*}\DD(R) \rangle =
 \langle \zeta \otimes i_{*}\DD(R), \pi^{*}\DD(Y) \rangle $$
\end{thm} 
 
\begin{proof}
The case where $X$ and $Y$ are curves was considered in Theorem 1.2 of \cite{P-orbifold},
and the proof in our case is very similar. The fact that $\pi^*$ is fully faithful follows immediately
from the equality $(\pi_*\OO_X)^{\Z_2}=\OO_Y$ and the projection formula.
Similarly, to prove that $i_*$ is fully faithful it suffices to check $(Li^*i_*F)^{\Z_2}=F$. We have
a canonical exact triangle
$$F\otimes N^{\vee}[1]\to Li^*i_*F\to F\to\ldots$$
compatible with $\Z_2$-action, where $N^{\vee}=\OO_X(-R)|_R$ is the conormal bundle. 
It remains to observe that $\Z_2$ acts on $N^{\vee}$ by multiplication with $-1$.

Now let $F\in \DD(Y)$ and $G\in \DD(R)$ be some objects. Then we have
$$\Hom_{\Z_2}(\pi^*(F), \zeta\otimes i_*(G))\simeq \Hom_{\Z_2}(Lj^*F,\zeta\otimes G)=0$$
which gives one of the required orthogonality conditions. On the other hand,
by Serre duality, denoting $d=\dim X$, we get
$$\Hom_{\Z_2}(i_*(G),\pi^*(F))^*\simeq \Hom_{\Z_2}(\pi^*(F),\om_X\otimes i_*(G)[d])\simeq
\Hom_{\Z_2}(Lj^*F,i^*\om_X\otimes G[d]).$$
Note that $\Z_2$ acts nontrivially on $i^*\om_X\simeq \om_Y\otimes N^{\vee}$,
so the above Hom-space vanishes. 

Finally, we have to check that for every $F\in\DD_{\Z_2}(X)$ such that $\Hom_{\Z_2}(i_*\DD(R),F)=0$
or $\Hom_{\Z_2}(F,\zeta\otimes i_*\DD(R))=0$, 
lies in the essential image of $\pi^*:\DD(Y)\to\DD_{\Z_2}(X)$.
Note that by Serre duality, these two orthogonality conditions are equivalent. Assume that
$\Hom_{\Z_2}(F,\zeta\otimes i_*\DD(R))=0$. Then $\Z_2$ acts trivially on $i^*F$. Now the assertion
follows from the main theorem of \cite{Ter}.
\end{proof}

We can use the above Theorem as a setup for gluing stability conditions.
The situation seems to be especially nice when either $\DD(R)$ or $\DD(Y)$ admits an exceptional
collection (see Remark at the end of the previous section). The former possibiity occurs when
$X$ and $Y$ are curves and will be considered below. The latter possibility happens if, say, $Y$ is
a projective space. In particular, we derive the following result.

\begin{prop} Let $\pi:X\to\P^n$ be a smooth double covering ramified along a smooth hypersurface 
$j:R\hookrightarrow\P^n$. Assume we are given a reasonable stability $\si^R=(Z^R,P^R)$ on $\DD(R)$,
an $\Ext$-exceptional collection $(E_0,\ldots,E_n)$ on $\P^n$, and a set of vectors $v_0,\ldots,v_n$
in the upper half-plane such that $j^*E_i\in P^R(>1)$ for $i=0,\ldots,n$.
Then there exists a reasonable stability $\si=(Z,P)$ on $\DD_{\Z_2}(X)$ with
$$P(0,1]=[ i_*P^R(0,1], \pi^*E_0,\ldots,\pi^*E_n ],$$
$$Z(E)=v_0 x_0(R\pi_*(E(R))^{\Z_2})+\ldots+v_n x_n(R\pi_*(E(R))^{\Z_2})-
Z^R((i^*E\otimes N)^{\Z_2}),$$
where $x_0,\ldots,x_n:K_0(\P^n)\to\Z$ are the coordinates dual to the basis $([E_i])$.
\end{prop} 

\begin{proof}
This stability is obtained by gluing with respect to the semiorthogonal decomposition
\begin{equation}\label{semiorth-dec}
\DD_{\Z_{2}}(X) = \langle \pi^{*}\DD(Y), i_{*}\DD(R) \rangle.
\end{equation}
It exists by Theorem \ref{reason-glue-thm}, where $a<1$ should be taken bigger than all of the phases
of the vectors $v_i$ (see Remark after Theorem \ref{reason-glue-thm}).
To get the formula for the central charge we note that for $E\in D_{\Z_2}(X)$ one has
$$\rho_2(E)=i^!(E)^{\Z_2}\simeq (i^*E\otimes N)^{\Z_2}[-1],$$
$$\la_1(E)=R\pi_*(E(R))^{\Z_2}.$$
\end{proof}

For example, if $X\to\P^3$ is a double covering ramified along a smooth surface $S\sub\P^3$
then we can consider stabilities on $S$ constructed in \cite{AB}. Choosing an appropriate
$\Ext$-exceptional collection on $\P^3$ and using the above result we get
examples of stabilities on $\DD_{\Z_2}(X)$. 

\section{Double coverings of curves}

In section we will consider the case when $X$ and $Y$ are curves. In this case
the ramification divisor $R$ consists of points $p_1,\ldots,p_n$, and the category
$\DD(R)$ is generated by the orthogonal exceptional objects $\OO_{p_1},\ldots,\OO_{p_n}$.
Recall that the category $\DD(X)$ has a standard stability condition $\si_{st}$ with
$Z_{st}=-\deg+i\rk$ and $P_{st}(0,1]=\Coh(X)$. There is an induced stability condition on
$\DD_{\Z_2}(X)$ with the heart $\Coh_{\Z_2}(X)$
that we still denote by $\si_{st}$ (see \cite{MMS}).

\begin{lem}\label{simple-obj-lem} 
Let $E$ be an endosimple object of the category $\DD_{\Z_2}(X)$ (i.e., $\Hom(E,E)=k$).
Then for some $n\in\Z$ the object $E[n]$ is one of the following types:
\begin{enumerate}
\item a vector bundle;
\item the sheaf $\OO_{\pi^{-1}(y)}$ for $y\in Y$;
\item the sheaf $\zeta\otimes\OO_{2p_i}$ for some $i\in\{1,\ldots,n\}$;
\item the sheaf $\OO_{p_i}$ for some $i$;
\item the sheaf $\zeta\otimes\OO_{p_i}$ for some $i$.
\end{enumerate}
\end{lem}

\begin{proof}
The category $\Coh_{\Z_2}(X)$ has cohomological dimension $1$, so every
indecomposable object in $\DD_{\Z_2}(X)$ has only one nonzero cohomology.
Thus, we can assume that $E$ is a $\Z_2$-equivariant coherent sheaf.
Furthermore, since the torsion part of such a sheaf splits as a direct summand,
it is enough to consider the case when $E$ is an indecomposable torsion sheaf.
Then the support of $E$ is either $\pi^{-1}(y)$, where $y\in Y\setminus R$,
or $\{p_i\}$ for some $i\in\{1,\ldots,n\}$. In the former case $E\simeq\pi^*E'$,
where $E'$ is an endosimple sheaf on $Y$ supported at $y$, so $E'\simeq\OO_y$.
In the latter case there exists $m$ such that $E\simeq\OO_{mp_i}$ or
$E\simeq\zeta\otimes\OO_{mp_i}$. It remains to observe that for $m\geq 3$ the
sheaf $\OO_{mp_i}$ is not endosimple, since we can construct
its nonscalar endomorphism as the composition of natural maps
$$\OO_{mp_i}\to \OO_{(m-2)p_i}\to\OO_{mp_i}.$$
\end{proof}

We are going to construct explicitly some stability conditions on $\DD_{\Z_2}(X)$.
For this we will use a slight variation of the semiorthogonal decompositions considered in
Theorem \ref{semiorth-thm}.
Namely, for every partition of $\{1,\ldots,n\}$ into two disjoint subset $I$ and $J$ we have
\begin{equation}\label{semiorth-dec2}
\DD_{\Z_{2}}(X) = \lan\lan\zeta\otimes\OO_{p_j}\ |\ j\in J\ran, \pi^{*}\DD(Y), 
\lan \OO_{p_i}\ |\ i\in I\ran\ran.
\end{equation}

For a subset $I\sub\{1,\ldots,n\}$ let us denote by $\DD(I)\sub\DD_{\Z_2}(X)$ the full triangulated
subcategory generated by $\pi^*\DD(Y)$ and $\OO_{p_i}$ with $i\in I$.

\begin{lem}\label{coh-lem} 
For $I\sub\{1,\ldots,n\}$ set $\Coh(I):=\Coh_{\Z_2}(X)\cap\DD(I)$. Then $\Coh(I)$
is the heart of a $t$-structure on $\DD(I)$. The natural exact functor $\Coh(I)\to\Coh_{\Z_2}(X)$ 
gives an equivalence of $\Coh(I)$ with the full subcategory of $\Coh_{\Z_2}(X)$ consisting of
all successive extensions of sheaves 
in $\pi^*\Coh(Y)$ and equivariant sheaves supported on $\{p_i \ |\ i\in I\}$.
The category $\Coh(I)$ is Noetherian.
\end{lem}

\begin{proof} Note that an object $E\in\DD_{\Z_2}(X)$ belongs to $\DD(I)$ if and only if
$\Hom^*(\OO_{p_i},E)=0$ for each $i\not\in I$. Since the category $\Coh_{\Z_2}(X)$ has
cohomological dimension $1$, we have $E\simeq\oplus H^iE[-i]$, where $H^iE\in\Coh_{\Z_2}(X)$.
Therefore, $E\in\DD(I)$ if and only if $H^iE\in\DD(I)$ for every $i$. This immediately implies that
the standard $t$-structure restricts to a $t$-structure on $\DD(I)$ with $\Coh(I)$ as the heart.
We have an exact embedding $\Coh(I)\to\Coh_{\Z_2}(X)$, so $\Coh(I)$ is Noetherian.
Let $\FF\in\Coh(I)$. Then the torsion part (resp., torsion-free part) of $\FF$ is also in $\Coh(I)$.
Assume first that $\FF$ is an indecomposable torsion sheaf with the support at $p_i$ for
$i\not\in I$. Then the condition
$\Hom^*(\OO_{p_i},E)=0$ easily implies that $E\simeq\OO_{2np_i}$.
On the other hand, if $\FF$ is a vector bundle then we have $\Hom(\FF,\zeta\ot\OO_{p_i})=0$
for $i\not\in I$, which implies that the fiber of $\FF$ at $p_i$ has trivial $\Z_2$-action for
$i\not\in I$. Therefore, making appropriate elementary transformations at $p_i$ for $i\in I$
we can represent $\FF$ as an extension of a sheaf supported at $\{p_i \ |\ i\in I\}$ by
the pull-back of a vector bundle from $Y$ (cf. proof of Theorem 1.8 of \cite{P-orbifold}).
\end{proof}

Given a partition of $\{1,\ldots,n\}$ into three disjoint subsets $I^0$, $I^+$ and $I^-$
we obtain from \eqref{semiorth-dec2} a semiorthogonal decomposition
\begin{equation}\label{semiorth-dec3}
\DD_{\Z_2}(X)=\lan \lan \zeta\otimes\OO_{p_i}\ |\ i\in I^-\ran, \DD(I^0),
\lan \OO_{p_i}, i\in I^+\ran\ran.
\end{equation}

\begin{prop}\label{cover-prop} Fix a partition $\{1,\ldots,n\}=I^0\sqcup I^+\sqcup I^-$ and
a collection of positive integers $(n_i)$ for $i\not\in I^0$.

\noindent
(a) Let $Z:\NN(\DD_{\Z_2}(X))\to\C$ be a homomorphism, such that 
\begin{enumerate}
\item
$\Im Z(\OO_X)>0$, and
$Z(\OO_{\pi^{-1}(y)})\in \R_{<0}$ for any point $y\in Y$;
\item
$Z(\OO_{p_i}[-n_i])\in\hh$ for $i\in I^+$, and $Z(\zeta\otimes\OO_{p_i}[n_i])\in\hh$ for $i\in I^-$;
\item
$Z(\OO_{p_i})\in \R_{<0}$ and $Z(\zeta\otimes\OO_{p_i})\in\R_{<0}$ for $i\in I^0$, 
\end{enumerate}
where $\hh\sub\C$ denotes the union of the upper half-plane with $\R_{<0}$.
Then there exists a reasonable stability condition $\sigma$ with the central charge
$Z$ and the heart 
\begin{equation}\label{glued-heart-eq}
H(I^+,I^-;\nn)=[ [\zeta\otimes\OO_{p_i}[n_i]\ |\ i\in I^-], \Coh(I^0),
[ \OO_{p_i}[-n_i], i\in I^+] ],
\end{equation}
which is glued with respect to the semiorthogonal decomposition \eqref{semiorth-dec3}.
All the objects $\OO_{\pi^{-1}(y)}$ for $y\in Y$ are $\sigma$-semistable (of phase $1$). The objects
$\OO_{\pi^{-1}(y)}$ for $y\in Y\setminus \{p_i\ |\ i\in I^0\}$, as well as
$\OO_{p_i}$ for $i\in I^0\cup I^+$ and $\zeta\otimes\OO_{p_i}$ for $i\in I^0\cup I^-$,
are $\sigma$-stable.

\noindent
(b) Assume in addition that $n_i=1$ for all $i\not\in I^0$. Then 
all the objects $\OO_{p_i}$ and $\zeta\otimes\OO_{p_i}$ for $i\in\{1,\ldots,n\}$ are $\sigma$-stable.
%Local finiteness
\end{prop}

\begin{proof} (a) Using the orthogonalities
$$\Hom^{\leq 0}(\Coh(I^0),\OO_{p_i}[-n_i])=\Hom^{\leq 0}(\zeta\otimes\OO_{p_j}[n_j],\Coh(I^0))=
\Hom^{\leq 0}(\zeta\otimes\OO_{p_j}[n_j],\OO_{p_i}[-n_i])$$
for $i\in I^+$, $j\in I^-$, we get the glued heart $H=H(I^+,I^-;\nn)$ given by
\eqref{glued-heart-eq}.
Note that the restriction of $Z$ to $\NN(\pi^*\DD(Y))$ is determined by $Z(\OO_X)$ and 
by $Z(\OO_{\pi^{-1}(y)})$ for a point $y\in Y$. Thus, $\Im Z(\pi^*F)=c\rk(F)$ for
some positive constant $c$. Since $\Coh(I^0)$ is generated by extensions from $\pi^*\Coh(Y)$
and $\OO_{p_i}$ and $\zeta\otimes\OO_{p_i}$ for $i\in I^0$, we deduce that
$Z$ is a stability function on $H$. It is also easy to see that $0$ is an
isolated point of $\Im Z(H)$. Since $H$ is glued from Noetherian
hearts, it is also Noetherian, so Lemma \ref{Noeth-lem}(a) implies that 
the Harder-Narasimhan property is satisfied for $Z$. Thus, we have a stability condition 
$\sigma=(Z,P)$ with $P(0,1]=H$. By Proposition \ref{char-prop}(2), it is glued from
the induced stability on $\DD(I^0)$ and the exceptional objects
$\zeta\otimes\OO_{p_i}[n_i]$, $i\in I^-$ and $\OO_{p_i}[-n_i]$, $i\in I^+$. The fact that
$\sigma$ is reasonable follows from Lemma \ref{Noeth-lem}(b).
Note that $P(1)\sub H$ consists of successive extensions of sheaves of the form $\OO_{\pi^{-1}(y)}$,
$y\in Y$, and of $\OO_{p_i}$ and $\zeta\otimes\OO_{p_i}$ for $i\in I^0$.
The simple objects in $P(1)$ are the
sheaves $\OO_{\pi^{-1}(y)}$, $y\in Y\setminus\{p_i\ |\ i\in I^0\}$, and
$\OO_{p_i}$ and $\zeta\otimes\OO_{p_i}$
for $i\in I^0$, so all these objects are $\sigma$-stable. On the other hand,
Proposition \ref{char-prop}(iii) implies
that the above exceptional objects in the heart corresponding
to $i\in I^+\cup I^-$, are $\sigma$-stable. 

\medskip

\noindent (b) Let us denote 
$$\CC^+:=[\OO_{p_i}\ |\ i\in I^+]\sub\Coh_{\Z_2}(X),$$
$$\CC^-:=[\zeta\otimes\OO_{p_i}\ |\ i\in I^-]\sub\Coh_{\Z_2}(X).$$
From the definition of $H$ one can easily deduce that for every object $C\in H$ one has
\begin{align*}
& H^{-1}C\in\CC^-;
& H^1C\in\CC^+;
& H^0C\simeq H^0(F_{-1}\to F_0\to F_1), \text{where } F_0\in\Coh(I^0), F_{-1}\in\CC^-, F_1\in\CC^+.
\end{align*}
The last condition easily implies that $\Hom(\CC^+,H^0C)=\Hom(H^0C,\CC^-)=0$.

Now let us fix $i\in I^+$ and consider the object $E=\zeta\otimes\OO_{p_{i}}$.
Note that
$\zeta\otimes\OO_{p_i}$ belongs to $H$, as an extension of $\OO_{2p_i}$ by $\OO_{p_i}[-1]$.
Suppose we have
a short exact sequence 
$$0\to A\to E\to B\to 0$$
in $H$ with nonzero $A$ and $B$. 
Since $H^2A=H^{-2}B=0$, we derive that $H^1B=H^{-1}A=0$ and there is an exact sequence
\begin{equation}\label{long-ex-seq}
0\to H^{-1}B\to H^0A\to E\to H^0B\to H^1A\to 0
\end{equation}
in $\Coh_{\Z_2}(X)$. 
Note that since $E$ is a simple object of $\Coh_{\Z_2}(X)$ we have one of the following two cases:
(i) $H^0B\to H^1A$ is an isomorphism; (ii) $H^{-1}B\to H^0A$ is an isomorphism.
In the first case we obtain that $H^0B\in\CC^+$ which implies that $H^0B=0$.
Hence, in this case $B\in\CC^-[1]$, so $\Hom(E,B)=0$ which is a contradiction.
Now let us consider case (ii). We have $H^0A\in\CC^-$, hence $H^0A=0$. It follows that
$A=H^1A[-1]$, and $B=H^0B$ is an extension of $H^1A$ by $E$.
Since $\Hom(H^1A,B)=0$, this extension cannot split on any direct summands of 
$H^1A$, which implies that
$A\simeq\OO_{p_i}[-1]$ and $B\simeq \OO_{2p_i}$.
Since $Z(\OO_{p_i}[-1])$ has smaller phase then 
$Z(E)$, this shows that $\zeta\otimes\OO_{p_i}$ is stable.
Similarly one proves that all the objects $\OO_{p_i}$ for $i\in I^-$ are stable.
\end{proof}

In the case when all $n_i$'s are equal to $1$,
we denote the heart $H(I^+,I^-,\nn)$ considered in the above Proposition simply by
$H(I^+,I^-)$.

We have the following partial characterization of 
stability conditions constructed above. 

\begin{lem}\label{phase-lem}
Let $\sigma=(Z,P)$ be a stability condition such that
$\OO_{\pi^{-1}(y)}\in P(1)$ for all $y\in Y\setminus R$.

\noindent
(a) Assume that $\OO_{2p_i}\in P(1)$ for all $i$, and for every $i$
one of the following three conditions holds:
\begin{enumerate}
\item both $\OO_{p_i}$ and $\zeta\otimes \OO_{p_i}$ are $\sigma$-semistable of phase $1$;
\item $\OO_{p_i}$ is $\sigma$-semistable of phase $>1$;
\item $\zeta\otimes\OO_{p_i}$ is $\sigma$-semistable of phase $\leq 0$. 
\end{enumerate}
Assume in addition that for every line bundle $L$ on $Y$ one has $\pi^*L\in P(0,1]$.
Then $\sigma$ coincides
with one of the stability conditions constructed in Proposition \ref{cover-prop}.
The latter condition is uniquely determined by $Z$ and by the phases of $\OO_{p_i}$ and
$\zeta\otimes\OO_{p_i}$ for $i\in\{1,\ldots,n\}$.

\noindent
(b) Now assume that $\sigma$ is locally finite, and for all $i\in\{1,\ldots,n\}$ one has
$\OO_{p_i}\in P[1,2)$ and $\zeta\otimes\OO_{p_i}\in P(0,1]$.  Assume in addition that
either all objects $\OO_{\pi^{-1}(y)}$ for $y\in Y\setminus R$ are stable, or
$\Im Z(V)>0$ for every $\Z_2$-equivariant vector bundle $V$.
Then $\sigma$ coincides with one of stability conditions
constructed in Proposition \ref{cover-prop} with $I^-=\emptyset$,
$$I^+=\{ i\ |\ \Im Z(\OO_{p_i})<0\},$$
and all $n_i$'s equal to $1$. 
%In particular, for every such stability one has
%$V\in P(0,1)$ for every equivariant vector bundle on $X$. Also, $P(1)$ is the category of
%equivariant torsion sheaves supported on $Y\setminus\{p_i\ |\ \Im Z(\OO_{p_i})\neq 0\}$.
\end{lem}

\begin{proof} (a) 
Let $I^0$, $I^+$ and $I^-$
be the subsets of $i$ such that conditions (1), (2) and (3) hold, respectively.
Note that since we have nonzero maps $\OO_{p_i}\to \zeta\otimes\OO_{p_i}[1]$, the conditions
(2) and (3) (and therefore, the subsets $I^0$, $I^+$ and $I^-$) are mutually disjoint.
For each $i\in I^+$ (resp., $i\in I^-$)
there is a unique $n_i>0$ such that $\phi(\OO_{p_i})-n_i\in (0,1]$
(resp., $\phi(\OO_{p_i})+n_i\in (0,1]$). Then $Z$ satisfies the conditions of Proposition 
\ref{cover-prop}, so it remains to check that $H=H(I^+,I^-;\nn)\sub P(0,1]$.
Note that by definition, we have $\OO_{\pi^{-1}(y)}\in P(0,1]$ for all $y\in Y$;
$\OO_{p_i},\zeta\otimes\OO_{p_i}\in P(1)$ for $i\in I^0$; $\OO_{p_i}[-n_i]\in P(0,1]$ for
$i\in I^+$ and $\zeta\otimes\OO_{p_i}[n_i]\in P(0,1]$ for $i\in I^-$. It remains to show
that $\pi^*V\in P(0,1]$ for every vector bundle $V$ on $Y$. But such a vector bundle
can be presented as an extension of line bundles, so this follows from our assumption.

\medskip

\noindent
(b) It is enough the check that 
$P(0,1]\sub H=H(I^+,\emptyset)$ (where $I^0$ is the complement to $I^+$).
First, we observe that in this case all equivariant vector bundles are in $H$, as extensions of
direct sums of sheaves of the form $\zeta\otimes\OO_{p_i}$ by a sheaf in $\pi^*\Coh(Y)$.
Let $E$ be a $\sigma$-stable object in $P(0,1)$. Note that $E$ is endosimple.
Let us consider possibilities for $E$ listed in Lemma \ref{simple-obj-lem}.
Since $Z(\OO_{\pi^{-1}(y)})=Z(\zeta\otimes\OO_{2p_i})\in\R_{<0}$ and $E\in P(0,1)$,
we obtain that for some $m\in\Z$,
$E[m]$ is either a vector bundle, or isomorphic to $\OO_{p_i}[-1]$,
or to $\zeta\otimes\OO_{p_i}$. In the last two cases our assumptions on $\sigma$ imply that $m=0$, so $E\in H$. If $E[m]$ is a vector bundle then using the condition
$E\in P(0,1)$ we get
\begin{equation}\label{Hom-van-eq}
\Hom^{\leq -1}(E,\OO_{\pi^{-1}(y)})=\Hom^{\leq 0}(\OO_{\pi}^{-1}(y),E)=0.
\end{equation}
This implies that $m=0$, so $E\in H$. 
Next, let $E$ be a $\sigma$-stable object in $P(1)$. We can assume that $E$ is not isomorphic
to $\OO_{\pi^{-1}(y)}$ for $y\in Y\setminus R$ since these objects are in $H$. 
Assume that $E[m]$ is a vector bundle. Note that this case cannot occur if $\Im Z(V)>0$ for all equivariant
vector bundles, so we can assume that the objects $\OO_{\pi^{-1}(y)}$ for $y\in Y\setminus R$ are
stable. Then the vanishing \eqref{Hom-van-eq} still
holds, so we deduce again that $m=0$. The case when
$E[m]$ is either $\OO_{p_i}$, or $\zeta\otimes\OO_{p_i}$ (where $i\in I^0$) is also clear.
Note that for $i\in I^0$ we have $\OO_{p_i},\zeta\otimes\OO_{p_i}\in P(1)$. Hence, for
such $i$ the objects $\OO_{2p_i}$ and $\zeta\otimes\OO_{2p_i}$ are not $\sigma$-stable.
Now
assume that $E[m]\simeq \OO_{2p_i}$, where $i\in I^+$. Since $\OO_{2p_i}\in P(0,2)$ as an extension
of $\OO_{p_i}$ by $\zeta\otimes \OO_{p_i}$, this implies that $m=0$, so $E\in H$. Finally,
we observe that for $i\in I^+$ 
the object $\zeta\otimes\OO_{2p_i}$ is not semistable since it is an extension
of $\zeta\otimes\OO_{p_i}$ by $\OO_{p_i}$, where $\phi_{\min}(\zeta\otimes\OO_{p_i})<1$
and $\phi_{\max}(\OO_{p_i})>1$.
\end{proof}

Note that the classes $[\OO_X]$, $[\OO_{\pi^{-1}(y)}]$, and $[\OO_{p_i}]$, $i\in\{1,\ldots,n\}$,
form a basis in $\NN(\DD_{\Z_2}(X))$. Thus, we can define a norm on the vector space
$\Hom(\NN(\DD_{\Z_2}(X)),\C)$ by setting
$$||Z||=\max(|Z(\OO_X)|,\max_{E}|Z(E)|),$$
where $E$ runs over all  endosimple  torsion sheaves in $\Coh_{\Z_2}(X)$ (see Lemma
\ref{simple-obj-lem}). It is also convenient to set for $Z\in \Hom(\NN(\DD_{\Z_2}(X)),\C)$
$$v_Z:=Z(\OO_{\pi^{-1}(y)})\in\C.$$
Let us define an open subset $\ov{U}\sub \Hom(\NN(\DD_{\Z_2}(X)),\C)$ as the set
of central charges $Z$ satisfying the following assumptions:
\begin{enumerate}
\item
for every $\Z_2$-equivariant line bundle $L$ on $X$ one has
$\det(Z(L),v_Z)>0$;
\item
for every $i=1,\ldots,n$ one has
$Z(\OO_{p_i})\not\in\R_{\leq 0}\cdot v_Z$, 
$Z(\zeta\otimes\OO_{p_i})\not\in\R_{\leq 0}\cdot v_Z$.
\end{enumerate}
Note that in the first condition it is enough to consider representatives in the cosets for
the subgroup $\pi^*\Pic(Y)\sub \Pic_{\Z_2}(X)$, so there is only finite number of inequalities 
to check (hence, $\ov{U}$ is open). Also, this condition implies that
$\det(Z(V),v_Z)>0$ for every equivariant vector bundle $V $on $X$, since they can be obtained
from line bundles by successive extensions.

\begin{lem}\label{num-lem}
\begin{enumerate}
\item
Let $Z:\NN(\DD_{\Z_2}(X))\to\C$ be a homomorphism such that $\Im Z(\OO_X)>0$,
$Z(\OO_{\pi^{-1}(y)})\in\R_{<0}$, and for every $i=1,\ldots,n$ one has
$Z(\OO_{p_i})\neq 0$ and $\Im Z(\OO_{p_i})\leq 0$.
%\in\C\setminus\R_{\geq 0}$.
Then there exists a constant $r>0$ such that 
for every $Z'\in \Hom(\NN(\DD_{\Z_2}(X)),\C)$ and every endosimple
object $E\in\DD_{\Z_2}(X)$ one has
$$|Z'(E)|\leq r\cdot ||Z'||\cdot |Z(E)|.$$
\item
The above conclusion also holds for $Z\in\ov{U}$.
\end{enumerate}
\end{lem}

\begin{proof} 
(1) Our conditions on $Z$ imply that $Z(E)\neq 0$ for every endosimple torsion
$\Z_2$-equivariant coherent sheaf $E$. Therefore, we can set
$$r_1=\max_E(|Z(E)|^{-1}),$$
where $E$ runs over all endosimple torsion sheaves.
If $E$ is such a sheaf then $|Z'(E)|\leq ||Z'||$, so the required inequality holds for $E$
provided $r\geq r_1$.
Now assume that $E$ is a $\Z_2$-equivariant vector bundle on $X$.
Then there exists an exact sequence of the form
$$0\to \pi^*E'\to E\to\oplus_i \zeta\otimes\OO_{p_i}^{m_i}\to 0,$$
where $0\le m_i\leq\rk(E)$.
Then
$$|Z'(E)|\leq |Z'(\pi^*E')|+n\rk(E)\cdot ||Z'||.$$
Note that 
\begin{equation}\label{E'-K0-eq}
[\pi^*E']=\rk(E)[\OO_X]+\deg(E')[\OO_{\pi^{-1}(y)}]
\end{equation}
in $\NN(\DD_{\Z_2}(X))$. 
Thus, we obtain
\begin{equation}\label{Z'-eq}
|Z'(E)|\leq ||Z'||\cdot [(n+1)\rk(E)+\deg(E')].
\end{equation}
On the other hand, from the above exact sequence we get
$$\Im Z(E)=\Im Z(\pi^*E')+\sum_i m_i\cdot \Im Z(\zeta\otimes\OO_{p_i}).$$
Since $\Im Z(\zeta\otimes\OO_{p_i})\geq 0$ and
$\Im Z(\pi^*E')=\Im Z(\OO_X)\cdot\rk(E)$, we deduce that
$$\rk(E)\leq \frac{|Z(E)|}{\Im Z(\OO_X)}.$$
Also, from \eqref{E'-K0-eq} we get
$$|\deg(E')Z(\OO_{\pi^{-1}(y)})|\leq |Z(\pi^*E')|+\rk(E)|Z(\OO_X)|\leq
|Z(E)|+(n+1)\rk(E)\cdot ||Z||.$$
Using our estimate for $\rk(E)$ we get that
$$\deg(E')\leq |Z(\OO_{\pi^{-1}(y)})|^{-1}\cdot [1+(n+1)\frac{||Z||}{\Im Z(\OO_X)}]\cdot |Z(E)|.$$
Therefore, from \eqref{Z'-eq} we obtain
$$|Z'(E)|\leq r_2 ||Z'||\cdot |Z(E)|,$$
where
$$r_2=\frac{n+1}{\Im Z(\OO_X)}+|Z(\OO_{\pi^{-1}(y)})|^{-1}\cdot [1+(n+1)\frac{||Z||}{\Im Z(\OO_X)}].$$
It remains to set $r=\max(r_1,r_2)$.

\noindent
(2) The subset $\ov{U}\sub \Hom(\NN(\DD_{\Z_2}(X)),\C)$ is stable under composition with 
rotations of $\C$ and with automorphisms of $\NN(\DD_{\Z_2}(X))$ given by tensoring
with an equivariant line bundle $L$. Also, the norms $||\cdot ||$ and
$Z'\mapsto ||Z'\circ (\otimes L)||$ on the finite-dimensional vector space
$\Hom(\NN(\DD_{\Z_2}(X)),\C)$ are equivalent, while composing with a rotation of $\C$ does not
change the norms. Therefore, we can modify $Z$ using these operations before checking the
required inequalities. Rotating $Z$ we can assume that $v_Z\in\R_{<0}$. Next,
let $I\sub\{1,\ldots,n\}$ be the set of $i$
such that $\Im Z(\OO_{p_i})>0$. Taking $L=\OO(\sum_{i\in I}p_i)$ we will have
$$L\otimes\OO_{p_i}\simeq \begin{cases}\zeta\otimes\OO_{p_i}, & i\in I \\ \OO_{p_i}, & i\not\in I.
\end{cases}$$
Therefore composing $Z$ with tensoring by $L$ we get the situation considered in (1).
\end{proof}

Recall that for every point $\sigma\in\Stab_{\NN}(\DD)$ a neighborhood of $\sigma$ in 
$\Stab_{\NN}(\DD)$ is homeomorphic to a neighborhood of the corresponding central
charge in the linear subspace $W^{\NN}_{\sigma}\sub \Hom(\NN(\DD),\C)$. A stability condition
$\sigma$ is called {\it full} if $W^{\NN}_{\sigma}= \Hom(\NN(\DD),\C)$.
The above Lemma implies that every stability condition with the central charge in the set
$\ov{U}$ is full.

\begin{thm}\label{cover-thm} Let $U\sub\Stab_{\NN}(\DD_{\Z_2}(X))$ denote the set of
locally finite stability conditions $\sigma=(Z,P)$ such that 
\begin{enumerate}
\item
$\OO_{\pi^{-1}(y)}$ is stable of phase $\phi_{\sigma}$ for every  $y\in Y\setminus R$;
\item
$\OO_{p_i},\zeta\otimes\OO_{p_i}$ are semistable with the phases in 
$(\phi_{\sigma}-1,\phi_{\sigma}+1)$
for all $i=1,\ldots,n$.
\end{enumerate}
Then every point in $U$ is obtained from one of the stability conditions described in Proposition
\ref{cover-prop} with $I^-=\emptyset$ and all $n_i=1$
by the action of an element of $\R\times\Pic_{\Z_2}(X)$, where $\R$ acts
on $\Stab_{\NN}(\DD_{\Z_2}(X))$ by rotations (shifts of phases).
The subset $U$ is open in $\Stab_{\NN}(\DD_{\Z_2}(X))$.
The natural map $U\to\ov{U}$ is a universal covering of $\ov{U}$, and $\ov{U}=U/\Z$, where
$1\in\Z$ acts on the stability space by shifting phases by $2$. Furthermore,
$U$ is contractible.
\end{thm} 

\begin{proof} 
{\bf Step 1}. If $\sigma=(Z,P)\in U$ then
$\sigma$ is obtained from one of the stability conditions described in Proposition
\ref{cover-prop} with $I^-=\emptyset$ and all $n_i=1$
by the action of an element of $\R\times\Pic_{\Z_2}(X)$.
Indeed, by rotating $\sigma$ we can assume that $\phi_{\sigma}=1$.
Now using tensoring with an appropriate equivariant line bundle we can assume that
$\Im Z(\OO_{p_i})\leq 0$ for all $i$. 
It remains to apply Lemma \ref{phase-lem}(b). 

Note that this step implies that for $\sigma=(Z,P)\in U$ one has $Z\in \ov{U}$.

\noindent
{\bf Step 2}. Let $U'$ be the preimage of $\ov{U}$ in $\Stab_{\NN}(\DD_{\Z_2}(X))$. Then
the projection $U'\to\ov{U}$ is a covering map. This is checked exactly as in
Proposition 8.3 of \cite{Bridgeland-K3} using Lemma \ref{num-lem}(b).

\noindent
{\bf Step 3}. $U$ is open in $\Stab_{\NN}(\DD_{\Z_2}(X))$. Let $\sigma_0=(Z_0,P_0)\in U$.
We have to prove that any stability $\sigma=(Z,P)$, sufficiently close to $\sigma_0$, is still in $U$.
Using rotations it is enough to consider the case when $Z(\OO_{\pi^{-1}(y)})\in\R_{<0}$.
By Step 1 we can assume that $\sigma_0$ is a stability arising in Proposition \ref{cover-prop}
with $I^-=\emptyset$ and all $n_i$'s equal to $1$.
For a $\Z_2$-equivariant line bundle $L$ and
a stability condition $\sigma'=(Z',P')$ we denote by $\sigma'\otimes L$ the stability condition
with central charge $E\mapsto Z'(E\otimes L^{-1})$ and the heart $P'(0,1]\otimes L$.
It is enough to check that $\sigma=\sigma'\otimes L$, where $\sigma'$ is one of stability conditions 
from Proposition \ref{cover-prop} (with $I^-=\emptyset$ and $n_i=1$). 
Let us set $L=\OO_X(\sum_{i\in I(+)}p_i)$, where $I(+)=\{i \ |\ \Im Z(\OO_{p_i})>0\}$.
We claim that the central charge
$Z'(E):=Z(E\otimes L)$ satisfies the assumptions of Proposition \ref{cover-prop} with 
$I^+=\{i \ |\ \Im Z(\OO_{p_i})\neq 0\}$, $I^-=\emptyset$ and all $n_i=1$.
Indeed, first, note that $Z'(\OO_{\pi^{-1}(y)})=Z(\OO_{\pi^{-1}(y)})\in\R_{<0}$, and
$Z'(\OO_X)=Z(L)$ is in the upper-half plane, provided $\sigma$ is close
enough to $\sigma_0$. Next, using the fact that 
$$\OO_{p_i}\otimes L\simeq\begin{cases}\OO_{p_i}, & i\not\in I(+),\\
\zeta\otimes\OO_{p_i}, & i\in I(+)\end{cases}$$ 
one checks the remaining assumptions.
Therefore, by Proposition \ref{cover-prop}, there exists a stability condition $\sigma'$ with
the central charge $Z'$ and the heart $H(I^+,\emptyset)$. Now we claim
that $\sigma=\sigma'\otimes L$. Since the corresponding central charges are the same,
by Proposition \ref{FormalDistance}(a), it remains to check that 
$H(I^+,\emptyset)\otimes L\sub P(-1,2]$. It is easy to see that 
\begin{eqnarray}\label{cover-heart-eq}
H(I^+,\emptyset)\otimes L =&[ \OO_X(\sum_{i\in I(+)}p_i)\otimes\pi^*\Coh(Y), \nonumber\\
&[\zeta\otimes\OO_{p_i}\ |\ i\not\in I(+)], [\OO_{p_i}\ |\ i\not\in I(-)], \nonumber\\
&[\zeta\otimes\OO_{p_i}[-1]\ |\ i\in I(+)], [\OO_{p_i}[-1]\ |\ i\in I(-)] ],
\end{eqnarray}
where $I(-)=\{i \ |\ \Im Z(\OO_{p_i})<0\}$.
Hence,
$$H(I^+,\emptyset)\otimes L\sub T_0:=[ P_0(0,1], 
[ \OO_{p_i},\OO_{p_i}[-1],\zeta\otimes\OO_{p_i}[-1]\ |\ i=1,\ldots,n] ].$$
Furthermore, we have $\OO_{p_i}\in P_0[1,2)$ and $\zeta\otimes \OO_{p_i}\in P_0(0,1]$.
Hence, we have $T_0\sub P_0(-1+\eps,2-\eps)$ for some $\eps>0$ depending only on $\sigma_0$.
Thus, for $d(P,P_0)<\eps$ we obtain
$$H(I^+,\emptyset)\otimes L\sub P_0(-1+\eps,2-\eps)\sub P(-1,2]$$
as required.

\noindent
{\bf Step 4}. $U$ is closed in $U'$. More precisely, we claim that $U$ coincides with the set of
$\sigma\in U'$ such that $\OO_{\pi^{-1}(y)}$ is semistable of phase $\phi_{\sigma}$ 
for every $y\in Y\setminus R$, and for every $i\in\{1,\ldots,n\}$ the objects
$\OO_{p_i}$ and $\zeta\otimes\OO_{p_i}$ are semistable with the phases in
$[\phi_{\sigma}-1,\phi_{\sigma}+1]$.
(recall that the set of stability conditions such that a given object $E$ is semistable is closed).
Indeed, given such $\sigma=(Z,P)$, by rotating it and using tensoring with an equivariant line bundle
we can assume that $\phi_{\sigma}=1$, and
$\Im Z(\OO_{p_i})\leq 0$ for all $i$. Note that the condition $Z\in\ov{U}$ implies that
the phase of $\OO_{p_i}$ (resp., $\zeta\otimes\OO_{p_i}$) is in $[1,2)$ (resp., in $(0,1]$)
for every $i$, and 
$\Im Z(V)>0$ for every $\Z_2$-equivariant vector bundle $V$.
Hence, by Lemma \ref{phase-lem}(b), $\sigma$ is obtained by the construction of Proposition
\ref{cover-prop}, which implies that $\OO_{\pi^{-1}(y)}$ is stable for every $y\in Y\setminus R$.
It remains to note that for $\si\in U'$ the phases of $Z(\OO_{p_i})$ and of $Z(\zeta\ot\OO_{p_i})$ never
equal $\phi_{\sigma}\pm 1$.

Combining Steps 2, 3 and 4 we obtain that $U\to\ov{U}$ is a covering map.

\noindent
{\bf Step 5}. Assume $\sigma_1, \sigma_2\in U$ have the same central charge $Z$. Then $\sigma_2$
is obtained from $\sigma_1$ by a shift of phase in $2\Z$. Indeed, applying such a shift we
can assume that $\phi_{\sigma_1}=\phi_{\sigma_2}$. Furthermore, applying a rotation 
and tensoring with a line bundle, we reduce to the situation $\phi_{\sigma_1}=1$ and
$\Im Z(\OO_{p_i})\leq 0$ for all $i$. By Lemma \ref{phase-lem}(b), in this case the hearts of
$\sigma_1$ and $\sigma_2$ are the same.

\noindent
{\bf Step 6}. It remains to show that $U$ is contractible. 
We have a free action of $\R$ on $U$ by the
shift of phase, so it is enough to consider the section of this action consisting of 
$\sigma\in U$ with $\phi_{\sigma}=1$. In other words, we have to consider the subset of
$\ov{U}$ consisting of $Z$ with $v_Z=Z(\OO_{\pi^{-1}(y)})\in\R_{<0}$.
A homomorphism $Z$ in this subset is determined by the following contractible data:
\begin{enumerate}
\item $v_Z\in\R_{<0}$;
\item for every $i\in\{1,\ldots,n\}$, $Z(\OO_{p_i})\in\C\setminus(\R_{\geq 0}\cup (v_Z+R_{\leq 0}))$;
\item $Z(\OO_X)$ in some half-plane of the form $\Im z>c$.
\end{enumerate}
\end{proof}

\begin{remark} In the next section we will study more closely the case $g(Y)\geq 1$.
We will show that in this case the objects
$\OO_{\pi^{-1}(y)}$ for $y\in Y\setminus R$ are
automatically stable with respect to any stability on $\DD_{\Z_2}(X)$,
and will describe the entire space $\Stab_{\NN}(\DD_{\Z_2}(X))$.
\end{remark}

We conclude this section with one observation in the case where
$Y = \P^{1}$. 

\begin{prop}\label{exc-P1-prop}
Consider a stability $\sigma=(Z,P) \in U$, where $U$ is as in Theorem
\ref{cover-thm}. Assume that for every $i=1,\ldots,n$ the vectors $Z(\OO_{p_i})$ and
$Z(\OO_{2p_i})$ are linearly independent over $\R$.
Then some rotation of $\sigma$ is glued from an exceptional collection.
\end{prop}

\begin{proof} By Theorem \ref{cover-thm}, it is enough to check the same statement for
a stability $\sigma$ arising from the construction of Proposition
\ref{cover-thm} with $I^+=\{1,\ldots,n\}$, $I^-=I^0=\emptyset$ and $n_i=1$. We claim that  
in this situation for any sufficiently small $a>0$ 
the rotated stability $R_{-a}\sigma=(Z_a,P_a)$ is glued from an exceptional collection.
Indeed, if $a$ is small enough then we still have $\Im Z_a(\OO_{p_i})<0$ for all $i=1,\ldots,n$.
There is a unique $N\in\Z$ such that $\Im Z_a(\pi^*\OO(N))<0$ and $\Im Z_a(\pi^*\OO(N+1))>0$.
Consider the following full $\Ext$-exceptional collection on $\DD_{\Z_2}(X)$:
\begin{equation}\label{cover-exc-coll}
(\pi^*\OO(N)[1],\pi^*\OO(N+1),\OO_{p_1}[-1],\ldots,\OO_{p_n}[-1]).  
\end{equation}
There exists a glued stability condition with the heart generated by this exceptional collection and with
the central charge $Z_a$. To see that $R_a\sigma$ coincides with this stability condition,
by Proposition \ref{FormalDistance}(a), it is enough to check that all the objects of our exceptional collection lie in $P_a(-1,2]=P(-1-a,2-a]$. Recall that
$$P(0,1]=[ \pi^*\Coh(\P^1),[ \zeta\otimes\OO_{p_i},\OO_{p_i}[-1] \ |\ i=1,\ldots,n] ]$$
Thus, all the objects of the collection \eqref{cover-exc-coll}, except for $\pi^*\OO(N)[1]$,
lie in $P(0,1]\sub P(-1-a,2-a]$. 
Note that by our assumptions, the phases of $\OO_{p_i}[-1]$ are in $(0,1)$. Also, it is easy to see
that $\pi^*\OO(m)\in P(0,1)$ for every $m\in\Z$. The exact sequence
$$0\to\pi^*\OO(m-1)\to\pi^*\OO(m)\to\OO_{\pi^{-1}(y)}\to 0$$
in $P(0,1]$ shows that $\phi_{\max}(\pi^*\OO(m-1))\le\phi_{\max}(\pi^*\OO(m))$.

Now let us consider the exact sequence
$$0\to F\to\pi^*\OO(N)\to G\to 0$$
in $P(0,1]$, where $F$ is the maximal $\si$-destabilizing subobject in $\pi^*\OO(N)$. 
The corresponding long exact cohomology sequence in $\Coh_{\Z_2}(X)$ takes form
$$0\to H^0F\to\pi^*\OO(N)\to H^0G\to H^1F\to 0,$$
so either $H^0F=0$ or $H^0F$ is a line bundle.
In the former case we have
$F=H^1F[-1]\in [\OO_{p_i}[-1]\ |\ i=1,\ldots,n]$.
In the latter case we have $H^0F\simeq \pi^*\OO(m)(-\sum_{j\in J}p_j)$ for some $m\in\Z$ and
$J\sub\{1,\ldots,n\}$. Hence, in the derived category $H^0F$ can be viewed as an extension
of $\pi^*\OO(m)$ by $\oplus_{j\in J}\OO_{p_j}[-1]$.
Therefore, the phase of $F$ is bounded above by the maximum of the phases of
$Z(\OO_{p_i}[-1])$, $i=1,\ldots,n$ and of $Z(\pi^*\OO(m))$.
Note that we have a nonzero map from $\pi^*\OO(m)(-\sum_{i=1}^n 2p_i)\simeq\pi^*\OO(m-n)$
to $\pi^*\OO(N)$, so $m\le N+n$.
By making $a$ small enough we can assume that $N\leq 0$, so in this case we deduce that
$\pi^*\OO(N)\in P(0,\phi)$, where $\phi<1$ is the maximum of the phases of
$Z(\OO_{p_i}[-1])$, $i=1,\ldots,n$ and of $Z(\pi^*\OO(n))$. If in addition $a<1-\phi$ then
we get $\pi^*\OO(N)[1]\in P(1,2-a]\sub P(-1-a,2-a]$ as required.
\end{proof}

\section{Classification of stability conditions in the case $Y\not\simeq\P^1$}
\label{class-sec}

First, let us formulate an abstract version of Lemma 7.2 of \cite{GKR}.
We say that an object $E$ of an abelian (or triangulated category) is {\it rigid} if
$\Hom^1(E,E)=0$.

\begin{prop}\label{GKR-lem} 
Let $\AA$ be an abelian category of homological dimension $1$, and let
$$Y\to E\stackrel{f}{\to} X\to Y[1]$$ 
be an exact triangle in $D^b(\AA)$ with $E\in\AA$,
such that $\Hom^{\leq 0}(Y,X)=0$. Then $X=X_0\oplus X_1[1]$, where $X_0,X_1\in\AA$.
Let $f_0:E\to X_0$ be the map induced by $f$. Then 
\begin{enumerate} 
\item $\coker(f_0)$ and $X_1$ are rigid;
\item $\Hom^*(\coker(f_0),X_0)=\Hom^*(\coker(f_0),X_1)=0$;
\item $\Hom^0(\ker(f_0), X_0)=0$, and the map $\ker(f_0)\to X_1[1]$ induces an
isomorphism $$\Hom^0(X_1,X_1)\simeq\Hom^1(\ker(f_0),X_1).$$
\end{enumerate}
\end{prop}

\begin{proof} The first part of the proof of Lemma 7.2 in \cite{GKR}
gives the statement that $X=X_0\oplus X_1[1]$,
and 
$$\Hom^{\leq 0}(H^0(Y)\oplus\coker(f_0)[-1],X_0\oplus X_1[1])=0,$$
which implies (2). Since $X_0$ surjects onto $\coker(f_0)$, the natural map
$$\Hom^1(\coker(f_0),X_0)\to\Hom^1(\coker(f_0),\coker(f_0))$$ 
is surjective, so we deduce
that $\coker(f_0)$ is rigid. Next, we have an exact sequence
$$0\to X_1\to H^0(Y)\to\ker(f_0)\to 0$$
in $\AA$. Thus, the natural map
$$\Hom^1(H^0(Y),X_1)\to\Hom^1(X_1,X_1)$$ 
is surjective, and we obtain that $X_1$ is rigid. Using the same exact sequence we get (3).
\end{proof}

\begin{lem}\label{rigid-lem}
Assume $g(Y)\geq 1$. Then every rigid object
in $\Coh_{\Z_2}(X)$ is of the form 
$$\bigoplus_{i\in I}\OO_{p_i}^{\oplus m_i}\oplus
\bigoplus_{j\in J}\zeta\otimes\OO_{p_j}^{\oplus n_j},$$
where $I\cap J=\emptyset$.
\end{lem}

\begin{proof} The fact that $g(Y)\geq 1$ implies that $\om_Y$ has a nowhere vanishing
section. Hence, $\om_X$ has a $\Z_2$-invariant section vanishing only along $R\sub X$.
Therefore, for $F\in \Coh_{\Z_2}(X)$ such that $F$ is not supported on $R$ we have
$$\Hom^1(F,F)^*\simeq\Hom(F,F\otimes\om_X)\neq 0,$$
so $F$ cannot be rigid. Thus, any indecomposable rigid object should be 
supported at one of the ramification points. It is easy to check that the $\Z_2$-sheaf $\OO_{mp_i}$ 
(resp., $\zeta\ot\OO_{mp_i}$) is rigid only for $m=1$.
The assertion follows easily from this.
\end{proof}

Let us denote by $\DD_{p_i}\sub\DD_{\Z_2}(X)$ the triangulated subcategory generated
by equivariant sheaves supported on $p_i$.

\begin{lem}\label{gen1-lem} 
Assume $g(Y)\geq 1$, and let $\sigma=(Z,P)$ be a stability condition on $\DD_{\Z_2}(X)$. 
Then 
\begin{enumerate}
\item the object $\OO_{\pi^{-1}(y)}$ is $\sigma$-stable for every $y\in Y\setminus R$;
\item $\sigma$ restricts to a stability condition on $\DD_{p_i}$;
\item for any exact triangle $A\to \OO_X\to B\to A[1]$ in $\DD_{\Z_2}(X)$ with 
$\Hom^{\leq 0}(A,B)=0$ and nonzero $A$ and $B$, there exists $I\sub\{1,\ldots,n\}$ such that
either  $A=\OO_X(-\sum_{i\in I}m_ip_i)$ and $B=\oplus_{i\in I}\OO_{m_ip_i}$, where all
$m_i$'s are odd, or
$A=\oplus_{i\in I}\zeta\otimes\OO_{p_i}[-1]$ and $B=\OO_X(\sum_{i\in I}p_i)$;
\item there exists a $\sigma$-semistable equivariant line bundle.
\end{enumerate}
\end{lem}

\begin{proof} (1) Consider the triangle 
$$Y\to\OO_{\pi^{-1}(y)}\stackrel{f}{\to} X\to Y[1]$$ 
with $Y\in P(-\infty,t]$, $X\in P(t,+\infty)$, and assume that $X\neq 0$.
Then by Proposition \ref{GKR-lem}, we have $X=X_0\oplus X_1[1]$,
where $X_0$ and $X_1$ are equivariant coherent sheaves, and $X_1$ is rigid. 
By Lemma \ref{rigid-lem},  
$X_1$ is supported at $R$. Hence, $\Hom(X_1,X_1)=\Hom^1(\ker(f_0),X_1)=0$ (the isomorphism
comes from Proposition \ref{GKR-lem}(3)), which implies that $X_1=0$.
On the other hand, since $X_0\neq 0$, the condition $\Hom(\coker(f_0),X_0)=0$ (see
Proposition \ref{GKR-lem}(2)) implies that 
the map $f_0:\OO_{\pi^{-1}(y)}\to X_0$ is nonzero, so it is an embedding. But $\coker(f_0)$
is also rigid (see Proposition \ref{GKR-lem}(1)), so it is supported at $R$. Therefore, the extension
$$0\to \OO_{\pi^{-1}(y)}\to X_0\to \coker(f_0)\to 0$$
splits. Since $\Hom(\coker(f_0),X_0)=0$, this implies that $f_0$ is an isomorphism. 

\medskip

\noindent
(2)  Consider the triangle 
$Y\to E\stackrel{f}{\to} X\to Y[1]$ 
with $Y\in P(-\infty,t]$, $X\in P(t,+\infty)$, where $E$ is a sheaf supported at $p_i$,
and assume that $X\neq 0$. Applying Proposition \ref{GKR-lem} and Lemma \ref{rigid-lem}
again we see that $\coker(f_0)$ is supported at $R$, so we can write
$\coker(f_0)=C\oplus C'$, where $C$ is supported at $p_i$ and $C'$ is supported
at $R-p_i$. Since $\im(f_0)$ is supported at $p_i$, the extension
$$0\to\im(f_0)\to X_0\to\coker(f_0)\to 0$$
splits over $C'$. Since $\Hom(\coker(f_0),X_0)=0$, it follows that $C'=0$, so
$X_0$ is supported at $p_i$. Similarly, we have $X_1=A\oplus A'$, where
$A$ is supported at $p_0$ and $A'$ is supported at $R-p_i$. To prove that $A'=0$
we use the fact that the map $\ker(f_0)\to X_1$ factors through $A$,
so $\Hom^0(A',A')$ maps to zero under the induced map
$\Hom^0(X_1,X_1)\to\Hom^1(\ker(f_0),X_1)$. But the latter map is an isomorphism by
Proposition \ref{GKR-lem}(3), so we deduce that $A'=0$. Hence, $X$ is supported at $p_i$, and
so $Y$ is also supported at $p_i$.

\medskip

\noindent (3) By Proposition \ref{GKR-lem}, we have $B=B_0\oplus B_1[1]$, where
$B_0$ and $B_1$ are equivariant sheaves. The fact that $B_1$ is rigid (hence, torsion)
implies that $\Hom^1(\OO_X,B_1)=0$. Together with Proposition \ref{GKR-lem}(3) this
easily leads to $B_1=0$.
Let $f:\OO_X\to B=B_0$ be the map in our exact triangle. Assume first that $f$ is injective. Then 
$B$ is an extension of a rigid object
$\coker(f)$ by $\OO_X$, such that $\Hom^*(\coker(f),B)=0$. By Lemma \ref{rigid-lem},
we have
$$\coker(f)\simeq\bigoplus_{i\in I}P_i\oplus\bigoplus_{j\in J}Q_j,$$
where $P_i=\OO_{p_i}^{\oplus m_i}$, $Q_j=\zeta\ot\OO_{p_j}^{\oplus n_j}$.
Since $\Hom^1(P_i,\OO_X)=0$, the extension
$$0\to \OO_X\to B\to\coker(f)\to 0$$
splits over $P_i$, which implies that $P_i=0$ (since $\Hom(\coker(f),B)=0$).
Next, the map $\Hom(\coker(f),\coker(f))\to\Hom^1(\coker(f),\OO_X)$ induced by
the above extension is an isomorphism. Hence, for every $j$ the induced map
$\Hom(P_j,P_j)\to \Hom^1(P_j,\OO_X)$ is an isomorphism. The source of this map has
dimension $n_j^2$, while the target has dimension $n_j$, so we get that $n_j=1$.
This gives the required form of $A$ and $B$ in this case. 

Next, assume that $\ker(f)\neq 0$.
Then $\ker(f)$ is isomorphic to $\OO_X(-\sum_i m_ip_i)$, and $\im(f)\simeq\oplus_i \OO_{m_ip_i}$.
The condition $\Hom(\ker(f),B)=0$ implies that $\Hom(\ker(f),\im(f))=0$. Hence, all nonzero
$m_i$'s are odd. Let $I$ denote the set of $i$ for which $m_i\neq 0$.
The extension
$$0\to\im(f)\to B\to \coker(f)\to 0$$
still has the property that $\Hom^*(\coker(f),B)=0$. 
This implies that $\coker(f)$ is supported at $\{p_i\ |\ i\in I\}$. Hence, $B$ is also supported
at this set. The condition $\Hom(\ker(f),B)=0$ implies that all indecomposable direct sumands of $B$
are of the form $\OO_{np_i}$, where $i\in I$ and $n$ is odd (and there is at least one such factor
for every $i\in I$). Now the condition $\Hom^*(\coker(f),B)=0$ together with the
rigidity of $\coker(f)$ (using Lemma \ref{rigid-lem}) implies that $\coker(f)=0$.

\medskip

\noindent (4) It follows easily from (3) that one of the HN-factors of $\OO_X$
is a line bundle.
\end{proof}

\begin{lem}\label{O-lem} 
Assume $g(Y)\geq 1$, and let $\sigma=(Z,P)$ be a locally finite stability condition on
$\DD_{\Z_2}(X)$ such that $\OO_{2p_i}$ is semistable for every $i=1,\ldots,n$, and all 
$\OO_{\pi^{-1}(y)}$ for $y\in Y\setminus R$ are stable of phase $1$. Then
\begin{enumerate}
\item $\OO_{2p_i}\in P(1)$ for every $i$;
\item for every line bundle $M$ on $Y$ one has $\pi^*M\in P(0,1)$.
\end{enumerate}
\end{lem}

\begin{proof} (1) By Lemma \ref{gen1-lem}(4), we know that there exists a $\sigma$-semistable
equivariant line bundle $L$. Since for $y\in Y\setminus R$
we have nonzero morphisms $L\to \OO_{\pi^{-1}(y)}$
and $\OO_{\pi^{-1}(y)}\to L[1]$, it follows that $L\in P(t)$ for some $t\in [0,1]$. Furthermore,
we cannot have $t=0$ or $t=1$, since there can be only finite number of nonisomorphic
simple objects of phase $t$ with nonzero maps to (or from) $L$.
Now we have nonzero maps $L\to\OO_{2p_i}$ and $\OO_{2p_i}\to L[1]$ which implies that
the phase of $\OO_{2p_i}$ is in the interval $[t,t+1]$. But $Z(\OO_{2p_i})$ has phase $1$,
so the phase of $\OO_{2p_i}$ is equal to $1$.

\medskip

\noindent (2) Tensoring $\sigma$ with $M^{-1}$ we immediately reduce to the case $M=\OO_Y$.
First, we observe that for any equivariant line bundle $L$ one cannot have
$L\in P[1,+\infty)$ or $L\in P(-\infty,0]$. Indeed, this follows from the existence of nonzero morphisms
$L\to\OO_{\pi^{-1}(y)}$ and $\OO_{\pi^{-1}(y)}\to L[1]$ as in part (1). Let us consider the canonical
exact triangle 
$$A\to\OO_X\to B\to A[1]$$ 
with $A\in P[1,+\infty)$ and $B\in P(-\infty,1)$. By Lemma
\ref{gen1-lem}(3) and the above observation, we obtain that $A=\oplus_{i\in I}\zeta\otimes\OO_{p_i}[-1]$
for some $I\sub\{1,\ldots,n\}$. But this implies that for $i\in I$ one has 
$\zeta\otimes\OO_{p_i}\in P[2,+\infty)$, which contradicts to the existence of a nonzero morphism
from $\zeta\otimes\OO_{p_i}$ to $\OO_{2p_i}\in P(1)$. Therefore, $\OO_X\in P(-\infty,1)$.
Now consider the exact triangle 
$$C\to\OO_X\to D\to C[1]$$ 
with $C\in P(0,1)$ and $D\in P(-\infty,0]$.
Using Lemma \ref{gen1-lem}(3) we obtain that $D=\oplus_{i\in I}\OO_{m_ip_i}$, where all $m_i$'s
are odd. But we have a nonzero map $\OO_{2p_i}\to\OO_{p_i}\hookrightarrow\OO_{m_ip_i}$, 
which is a contradiction since $\OO_{2p_i}\in P(1)$. Hence, $D=0$ and
$\OO_X\in P(0,1)$.
\end{proof}

Let us set $\SS_i=\Stab(\DD_{p_i})$. This is a two-dimensional complex manifold that
we are going to describe explicitly below. Note that these spaces for different points $p_i$
are canonically isomorphic, so we will sometimes skip the index $i$ below.

\begin{prop}\label{point-prop} 
(a) Let $U^+\sub\SS$ (resp., $U^-\sub\SS$) denote the subset of $\sigma$ such that
$\OO_{p_i}$ (resp., $\zeta\otimes\OO_{p_i}$) is $\sigma$-stable. Let also
$W^+\sub\SS$ (resp., $W^-\sub\SS$) denote the subset of stabilities with respect to
which $\OO_{2p_i}$ (resp., $\zeta\otimes\OO_{2p_i}$) is semistable. Then $U^+$ and $U^-$
are open, $W^+$ and $W^-$ are closed, and
$$\SS=U^+\cup U^-=W^+\cup W^-.$$ 
The subset $W^+\cap W^-$ is contained in $U^+\cap U^-$ and consists of $\sigma$ such that
$\OO_{p_i}$ and $\zeta\ot\OO_{p_i}$ are stable of the same phase.
The subset
$U^+\cap U^-\cap W^+$ is characterized in $U^+\cap W^+$ by the condition
$\phi(\OO_{p_i})<\phi(\OO_{2p_i})+1$.
Similarly, the subset $U^+\cap U^-\cap W^-$ is characterized in $U^+\cap W^-$
by the inequality $\phi(\OO_{p_i})>\phi(\zeta\otimes\OO_{2p_i})-1$.

\noindent
(b) There is a holomorphic submersion $f_i:\SS_i\to\C$ such that
$\exp(\pi f_i)$ is equal to $Z(\OO_{2p_i})$, and $\Im(f_i)$ is equal to the phase of $\OO_{2p_i}$ on
$W^+$ and to the phase of $\zeta\otimes\OO_{2p_i}$ on $W^-$. 
%One has
%$f_i(\sigma\otimes\zeta)=f_i(\sigma)$, 
%where $\sigma\mapsto\sigma\otimes\zeta$ denotes the action on $\sigma$
%of the autoequivalence of $\DD_{p_i}$ given by tensoring with $\zeta$.
The action of the subgroup $\R\times\R^*_{>0}\sub\gltwoplus$
of rotations and rescalings induces an isomorphism of complex manifolds
$$\C\times \Sigma\widetilde{\to}\SS_i,$$
such that $f_i$ corresponds to the projection to the first factor,
where $\Sigma=f_i^{-1}(0)$ is a (noncompact) Riemann surface.

\noindent
(c) There is a well defined branch of $\frac{1}{\pi}\log Z(\OO_{p_i})$ 
(resp., $\frac{1}{\pi}\log Z(\zeta\otimes\OO_{p_i})$) on $U^+$ (resp., $U^-$)
that defines an isomorphism $\Sigma\cap U^+\simeq\C\setminus\R_{\ge 0}$
(resp., $\Sigma\cap U^-\simeq\C\setminus\R_{\ge 0}$). Under both these
isomorphisms $\Sigma\cap U^+\cap U^-$ is mapped to the subset
of $\C\setminus\R_{\ge 0}$ consisting of $z$ with $|\Im z|<1$.

\noindent
(d) The Riemann surface $\Sigma$ is simply connected and of parabolic type.
More precisely, there exists an isomorphism $\Sigma\simeq \C$
under which the function $Z(\OO_{p_i})$ on $\Sigma$ corresponds to
the function
$$z\mapsto \frac{1}{2}+\frac{1}{\sqrt{\pi}}\int_0^z e^{-t^2}dt.$$
\end{prop}

\begin{proof} (a) Recall that by Lemma \ref{simple-obj-lem}, the only endosimple objects in $\DD_{p_i}$
are $\OO_{p_i}$, $\zeta\ot\OO_{p_i}$, $\OO_{2p_i}$ and $\zeta\ot\OO_{2p_i}$. Given a stability
condition $\sigma=(Z,P)$, each subcategory $P(t)$ is generated by stable (hence, endosimple)
objects. In particular, stable objects generate $\DD_{p_i}$. This implies that we should have
at least two stable objects, and that one of the objects $\OO_{p_i}$ and $\zeta\ot\OO_{p_i}$ is
always stable (since the objects $\OO_{2p_i}$ and $\zeta\ot\OO_{2p_i}$ do not generate
$\DD_{p_i}$). Thus, we have $\SS=U^+\cup U^-$. 

Next, let us check
that either $\OO_{2p_i}$ or $\zeta\ot\OO_{2p_i}$ is always semistable, i.e.,
$\SS=W^+\cup W^-$.
If either $\OO_{p_i}$ or $\zeta\ot\OO_{p_i}$ is not stable then one of the objects
$\OO_{2p_i}$ and $\zeta\ot\OO_{2p_i}$ has to be stable 
(since there should be at least two stable objects).
Now assume that both $\OO_{p_i}$ and $\zeta\ot\OO_{p_i}$ are stable. 
If $\phi(\OO_{p_i})=\phi(\zeta\ot\OO_{p_i})$ then $\OO_{2p_i}$ and $\zeta\ot\OO_{2p_i}$ are both
semistable of the same phase. If $\phi(\OO_{p_i})>\phi(\zeta\ot\OO_{p_i})$ then
$\OO_{2p_i}$ is stable and $\zeta\ot\OO_{2p_i}$ is unstable (=not semistable). Similarly, 
if $\phi(\OO_{p_i})<\phi(\zeta\ot\OO_{p_i})$
then $\zeta\ot\OO_{2p_i}$ is stable and $\OO_{2p_i}$ is unstable. This proves that $\SS=W^+\cup W^-$.
Note in addition that $\OO_{2p_i}$ and $\zeta\ot\OO_{2p_i}$ cannot be both stable since
we have nonzero maps $\OO_{2p_i}\to \zeta\ot\OO_{2p_i}$ and $\zeta\ot\OO_{2p_i}\to\OO_{2p_i}$.

Let us classify stabilities such that $\OO_{p_i}$ is stable. The following $3$ cases
(not mutually exclusive) can occur: (i) $\OO_{2p_i}$ is stable; (ii) $\zeta\ot\OO_{2p_i}$ is stable; (iii) 
$\zeta\ot\OO_{p_i}$ is stable.

In case (i) we have $\phi(\OO_{p_i})>\phi(\OO_{2p_i})$ (since there is a nonzero map
$\OO_{2p_i}\to \OO_{p_i}$). The exact triangle
$$\OO_{p_i}[-1]\to\zeta\ot\OO_{p_i}\to\OO_{2p_i}\to\OO_{p_i}$$
shows that if $\phi(\OO_{p_i})\ge\phi(\OO_{2p_i})+1$ then $\zeta\ot\OO_{p_i}$ is not stable.
On the other hand, if $\phi(\OO_{p_i})<\phi(\OO_{2p_i})+1$ then one can easily check that 
$\zeta\ot\OO_{p_i}$ is stable. A stability condition in this case is uniquely determined by the phases and
central charges of $\OO_{p_i}$ and $\OO_{2p_i}$ that can be arbitrary such that
$\phi(\OO_{p_i})>\phi(\OO_{2p_i})$.

In case (ii) we have $\phi(\OO_{p_i})<\phi(\zeta\ot\OO_{2p_i})$ (because of the nonzero map
$\OO_{p_i}\to\zeta\ot\OO_{2p_i}$). The exact triangle
$$\zeta\ot\OO_{2p_i}\to \zeta\ot\OO_{p_i}\to \OO_{p_i}[1]\to\ldots$$
shows that if $\phi(\OO_{p_i})\le \phi(\zeta\ot\OO_{2p_i})-1$ then $\zeta\ot\OO_{p_i}$ is not
stable. One can also check that for $\phi(\OO_{p_i})>\phi(\zeta\ot\OO_{2p_i})-1$,
the object $\zeta\ot\OO_{p_i}$ is stable. A stability condition in case (ii) is uniquely determined by the
phases and central charges of $\OO_{p_i}$ and $\zeta\ot\OO_{2p_i}$ subject to the condition
$\phi(\OO_{p_i})<\phi(\zeta\ot\OO_{2p_i})$.

In case (iii) we have 
\begin{equation}\label{phases-in-1}
|\phi(\OO_{p_i})-\phi(\zeta\ot\OO_{p_i})|<1
\end{equation} 
(because of nonzero maps
$\OO_{p_i}\to\zeta\ot\OO_{p_i}[1]$ and $\zeta\ot\OO_{p_i}\to\OO_{p_i}[1]$).
One can easily check that if $\phi(\OO_{p_i})>\phi(\zeta\ot\OO_{p_i})$
(resp., $\phi(\OO_{p_i})<\phi(\zeta\ot\OO_{p_i})$) then $\OO_{2p_i}$ is stable and
$\zeta\ot\OO_{2p_i}$ is unstable (resp., $\zeta\ot\OO_{2p_i}$ is stable and
$\OO_{2p_i}$ is unstable). On the other hand, if $\phi(\OO_{p_i})=\phi(\zeta\ot\OO_{p_i})$
then both $\OO_{2p_i}$ and $\zeta\ot\OO_{2p_i}$ are semistable of the same phase.
A stability condition in case (iii) is uniquely determined by the phases and central charges of
$\OO_{p_i}$ and $\zeta\ot\OO_{p_i}$ subject to \eqref{phases-in-1}.

The above classification (complemented by a similar classification in the case where
$\zeta\ot\OO_{p_i}$ is stable) implies the required characterizations of $W^+\cap W^-$,
$U^+\cap U^-\cap W^+$ and $U^+\cap U^-\cap W^-$.

Note that the subsets $W^+$ and $W^-$ are closed by general properties of stability conditions.
It remains to check that $U^+$ and $U^-$ are open. We'll do this only for $U^+$ (the other case
will follow by applying the autoequivalence $\ot\zeta$).
Assume first that $\sigma=(Z,P)\in U^+\cap U^-$. Then there exists an interval $(t,t+\eta)$ with
$0<\eta<1$ such that all the objects $\OO_{p_i}$, $\zeta\ot\OO_{p_i}$, $\OO_{2p_i}$ and 
$\zeta\ot\OO_{2p_i}$ are in $P(t,t+\eta)$.
Hence, if $\sigma'=(Z',P')$ is sufficiently close to $\sigma$ then these four objects are still in
$P'(t',t'+\eta')$ for some $0<\eta'<1$. It follows from the above classification that 
in this case $\sigma'\in U^+\cap U^-$. Next, assume that 
$\sigma=(Z,P)\in U^+$ is such that $\OO_{2p_i}$
is stable and $\zeta\ot\OO_{p_i}$ is not stable. Then we have $\zeta\ot\OO_{p_i}\in P[\phi_0,+\infty)$,
where $\phi_0=\phi(\OO_{2p_i})$, and also $\zeta\ot\OO_{2p_i}$ is unstable. 
Hence, if $\sigma'=(Z,P')$ is sufficiently close to $\sigma$
then $\zeta\ot\OO_{p_i}\in P'(>\phi_0-1/3)$, $\OO_{2p_i}\in P'(\phi_0-1/3,\phi_0+1/3)$, and 
$\zeta\ot\OO_{2p_i}$ is $\sigma'$-unstable. Suppose that $\OO_{p_i}$ is not $\sigma'$-stable.
Then $\zeta\ot\OO_{p_i}$ and $\OO_{2p_i}$ have to be stable. But the above inclusions
show that the difference of phases of $\zeta\ot\OO_{p_i}$ and $\OO_{2p_i}$ is $<1$. Therefore,
$\OO_{p_i}$ is also $\sigma'$-stable by the above classification.
Finally, assume $\si\in U^+$ is such that $\zeta\ot\OO_{2p_i}$ is stable and $\zeta\ot\OO_{p_i}$
is not stable. Then setting $\phi_0=\phi(\zeta\ot\OO_{2p_i})$ we get 
$\zeta\ot\OO_{p_i}\in P(-\infty,\phi_0]$. The same argument as in the previous case shows that
this implies that $\OO_{p_i}$ is $\sigma'$-stable for $\sigma'$ close to $\sigma$.

\medskip

\noindent
(b) The fact that $f_i$ is well-defined and continuous follows from the fact that the phases
of $\OO_{2p_i}$ and $\zeta\ot\OO_{2p_i}$ agree on $W^+\cap W^-$. Since $\exp(\pi f_i)$ is
holomorphic by the definition of a complex structure on the stability space, it follows that
$f_i$ is holomorphic. 
Now let us consider 
the subgroup $\R\times\R^*_{>0}\sub\gltwoplus$ acting on the
stability space, where $(a,\la)\in\R\times\R^*_{>0}$ 
acts by the phase rotation $R_a$ combined with the rescaling of the central charge by $\la$.
Note that this action is compatible with the holomorphic action of this group on the central charges,
where we identify $\R\times\R^*_{>0}$ with $\C$ via $(a,\la)\mapsto \frac{\log(\la)}{\pi}+ia$.
Under this identification we have
$$f_i(z\cdot \sigma)=f_i(\sigma)+z.$$ 
This gives the required splitting $\C\times \Sigma\widetilde{\to}\SS_i$.

\medskip

\noindent
(c) The identifications of $\Sigma\cap U^+$, $\Sigma\cap U^-$ and $\Sigma\cap U^+\cap U^-$
follow easily from the proof of (a). Note that 
it is convenient to consider separately three regions in $\Sigma$ 
depending on whether $\si\in W^+\setminus W^-$, $\si\in W^-\setminus W^+$, or
$\si\in W^+\cap W^-$. In the latter case we have $\phi(\OO_{p_i})=\phi(\zeta\ot\OO_{p_i})$.
In the first case if in addition $\si\in U^+$ (resp., $\si\in U^-$) then $\phi(\OO_{p_i})>0$
(resp., $\phi(\zeta\ot\OO_{p_i})<0$), etc.

\medskip

\noindent
(d) As we have seen in (c), the function $Z(\OO_{p_i})$ restricts to
$\exp(\pi z)$ on $\Sigma\cap U^+\simeq \C\setminus\R_{\ge 0}$, hence, it has a logarthmic 
ramification above $0$ and $\infty$. On the other hand, since $Z(\OO_{p_i})=1-Z(\zeta\ot\OO_{p_i})$,
we see that the restriction of this function to $\Sigma\cap U^-$ has a logarithmic ramification above
$1$ and $\infty$. Now we can easily identify $\Sigma$ with the simply connected Riemann surface that
has $4$ logarithmic ramification points, two over $\infty$,
one over $0$ and one over $1$. Our result follows easily
from the Nevanlinna's classification of such surfaces (see \cite[sec. 45]{Nevanlinna}).
\end{proof}

\begin{cor}\label{fun-cor} 
The function $\delta_i:\SS_i\to\R$ given by
$$\delta_i(\sigma)=\begin{cases} \det(Z(\zeta\otimes\OO_{p_i}),Z(\OO_{2p_i})), & 
\zeta\otimes\OO_{2p_i} \text{ is } \sigma-\text{semistable},\\
0, & \OO_{2p_i} \text{ is } \sigma-\text{semistable} \end{cases}
$$
is continuous.
\end{cor}

Let us consider the submanifold $\Theta$ of 
$\SS_1\times\ldots\times\SS_n\times\C$
consisting of $(\sigma_1,\ldots,\sigma_n,z)$ such that $f_1(\sigma_1)=\ldots=f_n(\sigma_n)$. 
Note that by Proposition \ref{point-prop}(b), we have
$$\Theta\simeq\C\times\Sigma^n\times\C,$$
where the first factor corresponds to $f_i(\sigma_i)$.

\begin{thm}\label{cover-class-thm} 
Assume that $g(Y)\geq 1$. Then natural map
$$\rho:\Stab_{\NN}(\DD_{\Z_2}(X))\to \SS_1\times\ldots\times\SS_n\times\C:
\sigma\mapsto (\sigma|_{\DD_{p_1}},\ldots,\sigma|_{\DD_{p_n}},Z(\OO_X))$$
induces an isomorphism of $\Stab_{\NN}(\DD_{\Z_2}(X))$ with the open subset 
$\Theta^0\sub\Theta$ consisting of $(\sigma_1,\ldots,\sigma_n,z)$ such that
\begin{equation}\label{det-ineq}
\det(z,\exp(\pi f_1(\sigma_1)))+\sum_{i=1}^n\delta_i(\sigma_i)>0.
\end{equation}
The space $\Stab_{\NN}(\DD_{\Z_2}(X))$ is contractible.
\end{thm}

\begin{proof} Note that the map $\rho$ is well defined by Lemma \ref{gen1-lem}(2).
It is continuous and is compatible with 
the similar restriction map on the central charges and with the $\gltwoplus$-actions. 

\noindent
{\bf Step 1}. Let us check that the image of $\rho$ is contained in $\Theta^0$. 
The fact that it is contained in $\Theta$ follows immediately from the definitions, so
it remains to check that \eqref{det-ineq} holds whenever $\sigma_1,\ldots,\sigma_n$ 
are the restrictions of some $\sigma\in\Stab_{\NN}(\DD_{\Z_2}(X))$
to $\DD_{p_1},\ldots,\DD_{p_n}$. 
Recall that by Proposition \ref{point-prop}, for every $i\in\{1,\ldots,n\}$ either
$\OO_{2p_i}$ or $\zeta\otimes\OO_{2p_i}$ is $\sigma$-semistable.
Thus, by Lemmas \ref{gen1-lem}(1) and \ref{O-lem}(1),
rotating $\sigma$ and tensoring it with an appropriate line bundle, we can get a stability
with respect to which all
objects $\OO_{\pi^{-1}(y)}$ for $y\in Y\setminus R$ are stable of phase $1$, and all
objects $\OO_{2p_i}$ are semistable of phase $1$. 
Note that for such a stability inequality 
\eqref{det-ineq} is satisfied, as $\delta_i(\sigma_i)=0$ for all $i$ and
the first term in \eqref{det-ineq} is equal to $\Im Z(\OO_X)$ (recall that $\OO_X\in P(0,1)$ by
Lemma \ref{O-lem}). It remains
to check that the left-hand side of \eqref{det-ineq} for $\rho(\sigma)$
does not change upon tensoring $\sigma$ with an equivariant line bundle
(the $\gltwoplus$-invariance is clear). It is enough to compare the left-hand sides of
\eqref{det-ineq} for $\sigma$ and $\sigma'=\sigma\otimes \OO(-p_i)$, assuming
that all $\OO_{\pi^{-1}(y)}$ for $y\in Y\setminus R$ 
have phase $1$ and $\OO_{2p_i}$ is $\sigma$-semistable. Indeed, the central charge for
$\sigma'$, is given by $Z'(E)=Z(E(p_i))$, so $z=Z(\OO_X)$ will get replaced
by 
$$Z'(\OO_X)=Z(\OO(p_i))=Z(\OO_X)+Z(\zeta\otimes\OO_{p_i}),$$ 
so the first term $\Im Z(\OO_X)$
in \eqref{det-ineq} gets replaced by its sum with $\Im Z(\zeta\otimes\OO_{p_i})$. On the
other hand, since $\zeta\otimes\OO_{2p_i}$ is $\sigma'$-semistable, the term
$\delta_i(\sigma|_{\DD_{p_i}})=0$ gets replaced by
$$\delta_i(\sigma'|_{\DD_{p_i}})=\Im Z'(\zeta\otimes\OO_{p_i})=\Im Z(\OO_{p_i})=-\Im Z(\zeta\otimes
\OO_{p_i}).$$

\noindent
{\bf Step 2}. Up to a rotation and tensoring with a line bundle, every stability condition 
$\sigma\in\Stab_{\NN}(\DD_{\Z_2}(X))$ is obtained from the construction of Proposition \ref{cover-prop}.
Indeed, applying a rotation and tensoring with a line bundle 
we can assume that $\OO_{\pi^{-1}(y)}$ for all $y\in Y$ are $\sigma$-semistable of phase $1$.
Now we have to check the remaining conditions of Lemma \ref{phase-lem}(a).
By Lemma \ref{O-lem} we know that $\pi^*L$ is in $P(0,1)$ for every $L\in\Pic(Y)$.
Next, by Proposition \ref{point-prop}(a), for every $i$ the restriction of $\sigma$ to
$\DD_i$ belongs either to $W^+\cap W^-$, or to $(W^+\cap U^+)\setminus W^-$, or to
$(W^+\cap U^-)\setminus U^+$.
In the first case both $\OO_{p_i}$ and $\zeta\otimes\OO_{p_i}$ are stable of phase $1$. 
In the second case $\OO_{p_i}$ is stable of phase $>1$. Finally, in the third case
$\zeta\otimes\OO_{p_i}$ is stable of phase $\leq 0$ (this follows from Proposition \ref{point-prop}(a)).

\noindent
{\bf Step 3}. $\rho$ gives a bijection from $\Stab_{\NN}(\DD_{\Z_2}(X))$ to $\Theta^0$.
First, suppose we have two stability conditions $\sigma=(Z,P)$ and $\sigma'=(Z',P')$ such that 
$\rho(\sigma)=\rho(\sigma')$. Then $Z=Z'$ and the induced stability condition on $\DD_{p_i}$
for $\sigma$ and $\sigma'$ are the same. This implies that for every $i$, $\OO_{2p_i}$ is
$\sigma$-semistable if and only if it is $\sigma'$-semistable (of the same phase).
Therefore, rotating and tensoring with a line bundle we can
assume that $\OO_{\pi^{-1}(y)}$ for all $y$ are semistable of phase $1$ with respect to
both $\sigma$ and $\sigma'$. As we have seen in Step 2 this implies that conditions of Lemma
\ref{phase-lem}(a) are satisfied for $\sigma$ and $\sigma'$, which gives $\sigma=\sigma'$. 
On the other hand, given a point $(\sigma_1,\ldots,\sigma_n,z)\in\Theta^0$, using
the $\gltwoplus$-action and operations on $\Theta^0$ corresponding to tensoring with a line
bundle on $X$, we can assume that $\Im f_i(\sigma_i)=1$ and $\OO_{2p_i}$ is 
$\sigma_i$-semistable for every $i$. We can define the central charge $Z$ uniquely
so that $Z|_{\DD_i}=Z_i$ and $Z(\OO_X)=z$. Note that inequality \eqref{det-ineq} 
in this case takes form $\Im z>0$. Now using Proposition \ref{cover-prop} we can easily construct
the stability condition $\sigma$ on $\DD_{\Z_2}(X)$ with the central charge $Z$ and
the given restrictions $\sigma_i$ on $\DD_i$. 

\noindent
{\bf Step 4}. 
$\Stab_{\NN}(\DD_{\Z_2}(X))$ is connected. This follows from the continuity of gluing and Step 2.
More precisely, let us first show that all stabilities constructed in Proposition \ref{cover-prop}
belong to the same connected component. To this end we consider them as being
glued from $(\zeta\otimes\OO_{p_i}, i\in I^-)$ and $D(I^+\cup I^0)$. Now using Corollary \ref{exc-cor}
we can find a path from our stability to the one that has
the phases of all $\zeta\otimes\OO_{p_i}$'s for $i\in I^-$ in the interval $(-1,0)$, and the phases of
all $\OO_{p_i}$'s for $i\in I^+$ in the interval $(1,2)$ (in particular, we will have $n_i=1$ for 
all $i\in I^-\cup I^+$). 
By definition, such a stability belongs to the connected set $U$ considered in 
Theorem \ref{cover-thm}. Thus, the set $V$ of all stabilities constructed in Proposition \ref{cover-prop}
is connected. Therefore, for every equivariant line bundle $L$ the set $V\otimes L$ is still connected.
Since the standard stability is contained in all of these sets, the statement follows from Step 2.

\noindent
{\bf Step 5}. It follows from Step 4 that every $\sigma\in\Stab_{\NN}(\DD_{\Z_2}(X))$ is full.
Therefore, the projection from $\Stab_{\NN}(\DD_{\Z_2}(X))$ to the space of numerical central
charges is a local homeomorphism. This implies that 
$\rho:\Stab_{\NN}(\DD_{\Z_2}(X))\to\Theta^0$ is a local homeomorphism. Therefore,
by Step 3, it is a homeomorphism.

\noindent
{\bf Step 6}. It remains to prove $\Theta^0$ is contractible. 
By Proposition \ref{point-prop}, the space $\Theta$ can be identified with the
product $\C\times \Sigma^n\times\C$,
where the first factor corresponds to $f_i(\sigma_i)$. Let us consider
the projection $\Theta^0\to \C\times \Sigma^n$
obtained by omitting the last component. Each fiber of this projection is a half-plane.
Since the target is contractible, it follows that
$\Theta^0$ is also contractible.
\end{proof}

\end{document}